\documentclass[12pt, a4paper, reqno]{amsart}

\usepackage[top=3.5cm,bottom=3.5cm,left=2.5cm,right=2.5cm,heightrounded,bindingoffset=0mm]{geometry}

\usepackage[T1]{fontenc}
\usepackage[utf8]{inputenc}
\usepackage[english]{babel}
\usepackage{microtype}
\usepackage{lmodern}
\usepackage{amsmath,amsthm,amssymb,mathrsfs}
\usepackage{enumerate} 
\usepackage{graphicx}
\usepackage{svg}
\usepackage{color}
\usepackage{array}
\usepackage{graphicx}
\usepackage{tikz}

\usepackage{graphicx}

\newcommand{\bq}{\begin{equation}}
\newcommand{\eq}{\end{equation}}
\newcommand{\bqa}{\begin{eqnarray*}}
\newcommand{\eqa}{\end{eqnarray*}}

\usepackage{hyperref}
\hypersetup{colorlinks={true},linkcolor={blue},citecolor=blue}
\usepackage[nameinlink,noabbrev,capitalise]{cleveref}
\theoremstyle{plain}
\newtheorem{theo}{Theorem}[section]
\newtheorem{theorembis}{Theorem}

\newtheorem{prop}[theo]{Proposition}
\newtheorem{lemm}[theo]{Lemma}

\theoremstyle{definition}

\crefname{lemm}{Lemma}{Lemmas}
\crefname{prop}{Proposition}{Propositions}
\theoremstyle{remark}
\newtheorem{rema}[theo]{Remark} 

\DeclareSymbolFont{pletters}{OT1}{cmr}{m}{sl}
\DeclareMathSymbol{s}{\mathalpha}{pletters}{`s}

\def\bema{\begin{displaymath}}
  \def\enma{\end{displaymath}}
  
  \def\bemar{\begin{displaymath}\begin{array}{rcl}}
  \def\enmar{\end{array}\end{displaymath}}

  \def\R{\mathbb{R}}

  \def\E{\mathbf{E}}
  
  \def\N{\mathbb{N}}

  \def\S{\mathbb{S}}

  \def\ep{\varepsilon}
  
  \def\rem{\left\{ \begin{array}{lcl}\mathcal{O}\left(n^{\frac{d-3}{2}} \log(n)\right) &\mbox{for} & d\geq 3, \\
\mathcal{O}(n^{-\frac{1}{4}})&\mbox{for} & d=2\end{array}  \right.}

  \newcommand\comb[2]{ \left( \begin{array}{c} #1 \\ #2 \end{array}   \right) }

  \def\eq#1{\mathrel{\mathop{\kern0pt\sim}\limits_{#1}}}

\title[Almost Sure Uniform Convergence of Random Hermite Series]{Almost Sure Uniform Convergence of Random Hermite Series}

\author{Rafik Imekraz}
\address{Laboratoire MIA, EA 3165, La Rochelle Université, 
Avenue A. Einstein, 17031 La Rochelle, France
}
\email[R. Imekraz]{rafik.imekraz@univ-lr.fr}

\author{Micka\"el Latocca}
\address{
Laboratoire de Mathématiques et de Modélisation d'\'Evry (LaMME), UMR CNRS 8071\\
Université d'\'Evry\\
23 boulevard François Mitterrand, 91000 Évry-Courcouronnes, France
}
\email[M. Latocca]{mickael.latocca@univ-evry.fr}

\makeatletter
\def\l@subsection{\@tocline{2}{0pt}{2.5pc}{5pc}{}}
\makeatother
\makeatletter
\def\@tocline#1#2#3#4#5#6#7{\relax
  \ifnum #1>\c@tocdepth 
  \else
    \par \addpenalty\@secpenalty\addvspace{#2}%
    \begingroup \hyphenpenalty\@M
    \@ifempty{#4}{%
      \@tempdima\csname r@tocindent\number#1\endcsname\relax
    }{%
      \@tempdima#4\relax
    }%
    \parindent\z@ \leftskip#3\relax \advance\leftskip\@tempdima\relax
    \rightskip\@pnumwidth plus4em \parfillskip-\@pnumwidth
    #5\leavevmode\hskip-\@tempdima
      \ifcase #1
       \or\or \hskip 1em \or \hskip 2em \else \hskip 3em \fi%
      #6\nobreak\relax
    \dotfill\hbox to\@pnumwidth{\@tocpagenum{#7}}\par
    \nobreak
    \endgroup
  \fi}
\makeatother

\begin{document}
\newcommand{\mickael}[1]{{\color{teal} \textbf{M:} #1}}
\newcommand{\rafik}[1]{{\color{orange} \textbf{R:} #1}}

\begin{abstract} We continue the analysis of random series 
associated to the multidimensional harmonic oscillator $-\Delta+|x|^2$ on $\R^d$ with $d\geq 2$.
More precisely we obtain a necessary and sufficient condition to get the almost sure uniform convergence on the whole space $\R^d$. It turns out that the same condition gives 
the almost sure uniform convergence on the sphere $\S^{d-1}$ (despite $\S^{d-1}$ is a zero Lebesgue measure of $\R^d$).
From a probabilistic point of view, our proof adapts a strategy used by the first author for boundaryless Riemannian compact manifolds.
However, our proof requires sharp off-diagonal estimates of the spectral function of $-\Delta + |x|^2$. Such estimates are obtained using elementary tools.
\end{abstract}

\maketitle

\tableofcontents

\section{Introduction}

\subsection{Description of the results}

The present work is the continuation of \cite{randomh, imek-crusep,imek-aif} on random linear combinations of eigenfunctions of the multidimensional harmonic oscillator $-\Delta+|x|^2$ acting on $\R^d$.
The results of \cite[Section 4]{imek-aif} applied to the harmonic oscillator are only sufficient conditions to reach $L^p$ spaces with probability $1$. In the present paper, we shall give necessary and sufficient conditions.
In some sense, \cref{teo} gives final results accessible via the strategy used in \cite{imek-jep} for boundaryless compact Riemannian manifolds. In the case of the harmonic oscillator, we develop new estimates of its spectral function (see Proposition \ref{en} and Section \ref{proof-sp} for details).

Before going on the specific case of the harmonic oscillator and detailing the type of conclusion we are interested in, it is worthwhile to recall the seminal results that lead us to our study (we refer to the introduction of \cite{imek-jep} for a global account on the subject).

The subject had been introduced by Paley and Zygmund in \cite{paley1930} and aims to study
properties of random trigonometric functions, in particular they obtained a sufficient condition on a sequence of coefficients $(c_n)_{n\geq 1}$ in order to get the almost sure uniform convergence of random trigonometric series like
\begin{equation}\label{ran}
\sum \pm c_n \cos(nx).
\end{equation}
The Paley-Zygmund theorem had finally been improved by Salem and Zygmund in \cite{salem1954} to the following condition 
\begin{equation}\label{salem54}
\sum_{\ell\geq 2} \frac{1}{\ell \sqrt{\log(\ell)}} \Big(\sum_{n\geq \ell} |c_n|^2 \Big)^{1/2}<+\infty,
\end{equation}
but Salem and Zygmund understood that \eqref{salem54} was not necessary and hence the story did not end there! Actually, the problem of giving a necessary and sufficient condition had finally be solved by Dudley and Fernique \cite{dudley1967,fernique74} and required further developments of the theory of Gaussian processes. The problem had then been highly studied by Marcus and Pisier in the seminal book \cite{pisier1981} in which random trigonometric series like \eqref{ran} are replaced with natural random trigonometric series on compact groups.

In the last twenty years, there was a great motivation to consider random series, specially Gaussian random series, in the context of studying partial differential equations (like Schrödinger or wave equations) in a subcritical regime. For our purpose, we have been motivated by the three following papers (which are more focused on the linear counterpart of such considerations) by Ayache-Tzvetkov \cite{tzvetkov-ay}, Tzvetkov \cite{Tzvetkov4} and Burq-Lebeau \cite{burq-lebeau}. The initial problem by Paley and Zygmund can now be reformulated and generalized as follows: if one interprets the trigonometric functions $x\mapsto \cos(nx)$ as eigenfunctions of the Laplace-Beltrami operator on the torus $\mathbb{T}$, then we may replace such functions with eigenfunctions of an elliptic operator on a manifold $\mathcal{M}$ and we ask to give a necessary and sufficient condition of almost sure uniform convergence on $\mathcal{M}$ (one can also consider other type of convergences). 

For the almost sure convergence in $L^p$ for finite $p$, one may say that the problem is almost solved at least for some important linear models (see \cite{imek-crusep}). The case $p=\infty$ turns out to be much more complicated (see \cite{imek-aif,imek-jep,imek-hold}).

Let us now give a rigorous definition of the random series we shall use for the harmonic oscillator 
\begin{equation*}
  -\Delta + |x|^2, \quad x\in \mathbb{R}^d, \quad d\geq 2. 
\end{equation*}
We denote by $(E_{n})_{n\in \N}$ the sequence of its eigenspaces, namely 
\begin{equation*}
E_{n}=\ker((-\Delta+|x|^2) - (2n+d)) \qquad \mbox{so that}\qquad L^2(\R^d)=\bigoplus\limits_{n\in \N} E_n,
\end{equation*}
in which we denote $\N=\{0,1,\dots,\}$. We shall also denote $\N^\star=\{1,2,\dots,\}$ in the sequel.
Let $B(d,n)$ be a Hilbert basis of the eigenspace $E_{n}$, which we do not explicit.
Let us therefore write 
\begin{equation}\label{Enbasis}
B(d,n)=\{\varphi_{n,k},\; 1\leq k\leq \dim(E_n)\}.
\end{equation}
We merely assume that each $\varphi_{n,k}$ is real-valued. The following asymptotics for the dimension of the eigenspaces as $n\to \infty$ are well-known:  
\begin{equation}\label{dimEn}
\mbox{Card}B(d,n)=\dim(E_{n})\simeq n^{d-1}.
\end{equation}

For any real number $s$, we denote by $\mathcal{H}^s(\R^d)$ the Sobolev space associated to the harmonic oscillator consisting of all tempered distributions of the form
\begin{equation*}
f=\sum_{n\in \N}f_n ,\quad \mbox{with }f_n \in E_n,\qquad \sum_{n\in \N} (1+n)^s \|f_n\|_{L^2(\R^d)}^2<\infty.
\end{equation*}
We associate a random series to $f$, which we write $f^{G,\omega}$ and study under which conditions it almost surely belongs in $L^p(\mathbb{R}^d)$.

In order to define the suitable random series, we shall follow \cite{imek-crusep,imek-aif,imek-jep} which are indeed motivated by the multidimensional considerations of \cite{burq-lebeau}. 
We define\footnote{It turns out that many other methods of randomization can be used and proved to be equivalent (see \cite[Parts 11 and 12]{imek-jep}), but we shall not focus on these probabilistic aspects.}:
\begin{equation}\label{deffg}
  f^{G,\omega}=\sum\limits_{n\geq 1}f_n^{G,\omega},\qquad  f_n^{G,\omega}(x)=\frac{\|f_n\|_{L^2(\R^d)}}{\sqrt{\dim(E_{n})}} \sum_{k=1}^{\dim(E_n)}g_{n,k}(\omega) \varphi_{n,k}(x),\qquad
\end{equation}
in which the $(g_{n,k})$, $n\geq 0$, $1 \leq k \leq \dim E_n$ are i.i.d. real Gaussian random variables with distribution $\mathcal{N}(0,1)$. 

We note that the factor $\frac{1}{\sqrt{\dim(E_n)}}$ ensures that the Gaussian vector 
\begin{equation*}\frac{1}{\sqrt{\dim(E_n)}}(g_{n,1},\dots,g_{n,\dim(E_n)})
\end{equation*} has norm $1$ in $L^2(\Omega,\R^d)$.
A much deeper reason is that such a Gaussian vector is known to be, in some sense, comparable to a random vector whose distribution is uniform over the sphere $\S^{d-1}$ (see \cite[page 755, Proposition 23 and Remark 24]{imek-jep} and \cite[Page 58]{pis89}).

As it will be explained in Appendix \ref{anxA}, similar computations to the ones carried out in 
\cite[Theor\`eme 1.4] {imek-crusep} may prove the following result.
\newpage
\begin{theo}\label{thm.main-harmonic}
Let us fix $p\in [1,+\infty)$ and $f\in \bigcup\limits_{s\in \R}\mathcal{H}^s(\R^d)$. The following statements are equivalent:
\begin{enumerate}[(i)]
\item the condition $\sum\limits_{\ell\geq  1} \ell^{\frac{d}{2}-1} \Big(\sum\limits_{n\geq \ell}\frac{\|f_n\|_{L^2(\R^d)}^{2}}{n^{d/2}}\Big)^{p/2}<+\infty$ is satisfied, 
\item the Gaussian random series $\displaystyle\sum\limits_{n\in \N^\star} f_n^{G,\omega}$ almost surely converges in $L^p(\R^d)$.
\end{enumerate}
\end{theo}

The endpoint $p=\infty$ requires other ideas.
Similarly to Theorem \ref{thm.main-harmonic}, the natural remaining question is whether one can find a necessary and sufficient condition on a tempered distribution $f$ that ensures that the Gaussian random series $\sum f_n^{G,\omega}$ almost surely converges in $L^\infty(\R^d)$. 

Our main contribution completes \cite[Theorem 4.3 with $\alpha=1$]{imek-aif} and \cite[Theorem 2.6]{randomh} in an optimal way and reads as follows: 

\begin{theorembis}\label{teo}
Let $f\in \bigcup\limits_{s\in \R} \mathcal{H}^s(\R^d)$ be a tempered distribution. Then the following statements are equivalent:
\begin{enumerate}[(i)]
\item the following condition is satisfied 
\begin{equation}\label{sz}
\sum\limits_{\ell\geq 2} \frac{1}{\ell\sqrt{\log(\ell)}} \Big(\sum\limits_{n=\ell}^{+\infty} \frac{\|f_n\|_{L^2(\R^d)}^2}{n^{d/2}}\Big)^{1/2}<+\infty;
\end{equation}
\item the Gaussian random series $\displaystyle\sum\limits_{n\in \N} f_n^{G,\omega}$ almost surely converges in $L^\infty(\R^d)$;
\item the Gaussian random series $\displaystyle\sum\limits_{n\in \N} f_n^{G,\omega}$ almost surely converges in $L^\infty(\S^{d-1})$.
\end{enumerate}
\end{theorembis}

\begin{rema}
It is interesting to note that the previous result holds true despite $\S^{d-1}$ is a $0$-Lebesgue measure subset of $\R^{d}$. A similar phenomenon has been observed in \cite[Theorem 2]{imek-jep} for suitable Gaussian random series of eigenfunction of the Laplace-Beltrami operator on a boundaryless compact manifold. It can be seen from the proof of \cref{teo} that one can replace the uniform convergence on $\S^{d-1}$ with the uniform convergence on the geodesic $\S^1 \times \{0\}^{d-2}$.
\end{rema}
\begin{rema} More interestingly, Theorem \ref{teo} remains true by replacing $L^\infty(\S^d)$ by $L^\infty(\mathcal{R} \S^d)$ whatever the positive radius $\mathcal{R}$ is (this is also obtained by following the proof). In other words, if one fixes two different positive radii $\mathcal{R}_1 \neq \mathcal{R}_2$ then Theorem \ref{teo} shows that the two following statements are equivalent:
\begin{enumerate}
\item[(iv) ] the Gaussian random series $\displaystyle\sum\limits_{n\in \N} f_n^{G,\omega}$ almost surely converges in $L^\infty(\mathcal{R}_1\S^{d-1})$,
\item[(v) ] the Gaussian random series $\displaystyle\sum\limits_{n\in \N} f_n^{G,\omega}$ almost surely converges in $L^\infty(\mathcal{R}_2\S^{d-1})$.
\end{enumerate}
Such an equivalence is remarkable since the two spheres $\mathcal{R}_1 \S^{d-1}$ and $\mathcal{R}_2 \S^{d-2}$ are disjoint.
\end{rema}

\subsection{Sketch of the proof of \cref{teo}}

Let us explain the strategy of proof of \cref{teo}, and point out the main difficulties.  

To start with, the implication (\textit{i}) $\Longrightarrow$ (\textit{ii}) of \cref{teo} does not pose any real issue. It is a matter of generalizing a result by Salem and Zygmund \cite[Page 291, Theorem 5.1.5]{salem1954} in a modern framework. This is done by exploiting some estimates of \cite[Pages 2745-2750]{imek-aif}. It is also clear that the implication 
(\textit{ii}) $\Longrightarrow$ (\textit{iii}) is obvious.

Almost all the paper is devoted to prove the implication (\textit{iii}) $\Longrightarrow$ (\textit{i}). Let us recall a few ideas on the study of Gaussian processes (we refer to the seminal book \cite{ledoux} or the introduction \cite{imek-jep} for more details):
\begin{itemize}
\item Restricting Gaussian processes on compact subsets is often a good idea, as for the almost sure continuity, a Gaussian process is almost surely continuous on $\R^d$ if and only if it is almost surely continuous on each closed ball centered at $0$ (say of radius being integers). This is actually obvious since an intersection of countable events of probability $1$ has also probability $1$.
\item If a Gaussian process is defined on a compact subset having a natural group structure, e.g. a torus, and if the considered Gaussian process is stationary, then a necessary and sufficient condition for almost sure continuity is the so-called entropic condition (we refer for instance to \cite[Chapter 12]{ledoux}). 
\item Without any assumption of stationarity, one must find another approach. A general approach initiated by Fernique and highly developed by Talagrand is the theory of majorizing measures (see \cite[Chapter 11]{ledoux})). As explained in the introduction of \cite[page 755]{imek-jep}, one trick may sometimes be used, namely using comparison theorems of Gaussian processes (Slepian type results) to reduce the analysis to the stationary case. It turns out that such a strategy can be more amenable than the use of majorizing measures.
\item In any of the above strategies, it has been understood since the paper
\cite{dudley1967} that understanding the behavior of a Gaussian process $f^{G,\omega}$ is more or less equivalent to study its so-called Dudley pseudo-distance which we now introduce.
\end{itemize}

Let us now define the Dudley pseudo-distance denoted below by $\delta$ of the Gaussian process $f^{G,\omega}$. For any $(x,y)\in \R^d\times \R^d$, we set 
\begin{eqnarray}\nonumber
\delta(x,y)^2 & :=& \E[|f^{G,\omega}(x)-f^{G,\omega}(y)|^2] \\\nonumber
& =& \sum_{n\geq 0} \E[f_n^{G,\omega}(x)-f_n^{G,\omega}(y)|^2] \qquad \mbox{by independence}\\\label{dedel}
& =& \sum_{n\in \N} \|f_n\|_{L^2(\R^d)}^2 \delta_n(x,y)^2, 
\end{eqnarray}
where $\delta_n(x,y)$ is a pseudo-distance associated to the eigenspace $E_n$ and is computed as follows: 
\begin{equation}\label{deln}
\delta_n(x,y)^2=\frac{1}{\dim(E_n)}\sum_{k=1}^{\dim(E_n)} (\varphi_{n,k}(x)-\varphi_{n,k}(y))^2.
\end{equation}

In contrast to the analysis of \cite[Theorems 8 and 10]{imek-jep}, it does not seem possible to hope for two-sided estimates of $\delta_n(x,y)$ in $\R^d\times \R^d$. This is due to the fact that the eigenfunctions $\varphi_{n,k}$ of the harmonic oscillator $-\Delta+|x|^2$ are known to decay exponentially at infinity. This is why we may separate distinguish between the allowed region $\max\{|x|,|y|\}\ll \sqrt{n}$ and the forbidden region $\min\{|x|,|y|\}\gg \sqrt{n}$. We will indeed only work in the first of these regions. 

Before writing our asymptotic estimate of \eqref{deln}, we prefer first explain the intuition that leads us to the correct conjecture of this asymptotic. 
First, the Bernstein inequality for the operator $-\Delta+|x|^2$ (see for instance \cite[Theorem 5.3]{imek-aif}) shows the following inequalities for any $(x,y)\in \R^d\times \R^d$  
\begin{equation}\label{berns}
|\varphi_{n,k}(x)-\varphi_{n,k}(y)|\leq C\sqrt{n} \|\varphi_{n,k}\|_{L^\infty(\R^d)} |x-y|.
\end{equation}
Then, it is important to realize that the best available bound for $\|\varphi_{n,k}\|_{L^\infty(\R^d)}$ (see for instance \cite{KoTa05}) reads 
\begin{equation}
\label{bound-generic}  
  \|\varphi_{n,k}\|_{L^\infty(\R^d)} \lesssim n^{\frac{1}{2}\left(\frac{d}{2}-1 \right)},
\end{equation}
and is not good enough for our purposes. This stems from the fact that while \eqref{bound-generic} is a valid for any eigenfuctions, the quantity in which we are interested in \eqref{deln} does not depend on the choice of the Hilbert basis $(\varphi_{n,k})_{k}$ of $E_n$, as this can be directly checked in \eqref{deln} or by the rotational invariance of the Gaussian random vector used in \eqref{deffg}. An idea is thus to try to minimize the bound in \eqref{bound-generic} for a well-chosen basis $(\varphi_{n,k})$ of eigenfunctions, we indeed use the generic upper bound\footnote{The logarithmic factor comes from probabilistic arguments just as in \cite{burq-lebeau} and it is not known how to remove it in an explicit construction.} of \cite[Theorem 1.3]{PRT1}: for any $n\geq 2$, one has   
\begin{equation*}
 \|\varphi_{n,k}\|_{L^\infty(\R^d)}\lesssim \frac{\sqrt{\log(n)}}{n^{d/4}}, 
\end{equation*}
Therefore, combining this with \eqref{deln} and \eqref{berns}, we obtain the upper-bound 
\begin{equation*}
  \delta_n(x,y)\lesssim \frac{\sqrt{n}|x-y|}{n^{d/4}} \sqrt{\log(n)}.
\end{equation*}
Using \cite[Proposition 4.1]{imek-crusep}, we may bound \eqref{deln}: 
\begin{equation*}
  \delta_n(x,y)^2 \leq \frac{4}{\dim(E_n)} \sup\limits_{x\in \R^d} \sum_{k=1}^{\dim(E_n)} \varphi_{n,k}(x)^2\lesssim \frac{1}{n^{d/2}}, 
\end{equation*}
so that we finally arrive at 
\begin{equation*}
  \delta_n(x,y) \lesssim n^{-\frac{d}{4}}\min\left(1,\sqrt{n \log(n)}|x-y| \right).
\end{equation*}
The previous bounds are indeed true for any $(x,y)\in \R^d\times \R^d$ but as mentioned before, we will restrict the couples $(x,y)$ to a subregion of the allowed region in order to obtain lower bounds. More precisely, we shall prove in Proposition \ref{del-e} that, upon assuming $|x-y|\leq 1$ and $|x|=|y|=1$, the following two-sided estimates hold true
\begin{equation}\label{equ2}
\delta_n(x,y) \simeq n^{-\frac{d}{4}}\min\left(1,\sqrt{n}|x-y| \right).
\end{equation}
These estimates constitute the core of the proof of \cref{teo}. To get such estimates, one needs precise asymptotics of the so-called spectral function for the eigenvalue $2n+d$:
\begin{equation*}
e_{d,n}(x,y):=\sum_{k=1}^{\dim(E_n)} \varphi_{n,k}(x)\varphi_{n,k}(y), 
\end{equation*}
which, to the best of the knowledge of the authors is not already available in the literature. 
Here is a precise statement in which $\widetilde{J}_{\frac{d}{2}-1}$ denotes the normalized Bessel function:
\begin{equation}\label{norm-bess}
\widetilde{J}_{\frac{d}{2}-1}(t):=\Big(\frac{2}{t} \Big)^{\frac{d}{2}-1} J_{\frac{d}{2}-1}(t),\qquad \forall t>0
\end{equation}
with the convention $\widetilde{J}_{\frac{d}{2}-1}(0)=\frac{1}{\Gamma(\frac{d}{2})}$ (by continuous extension). We also refer to Section \ref{proof-sp} for graphical representations about the approximation \eqref{en-bessel} and the more precise formulas \eqref{red2a} and \eqref{red2b} for $d=2$.

\begin{prop}\label{en} 
Let us assume $d \geq 2$. The following asymptotic estimates\footnote{By comparison with \cite[line (8)]{canzani2015scaling}, it is maybe more reasonable to highlight the eigenvalue $2n+d$ and write the principal term as 
\begin{equation*}
\frac{n^{\frac{d}{2}-1}}{(2\pi)^{\frac{d}{2}}} \widetilde{J}_{\frac{d}{2}-1}(\sqrt{2n+d}|x-y|).
\end{equation*}
This is indeed equivalent to \eqref{en-bessel} because $\widetilde{J}_{\frac{d}{2}-1}$ is Lipschitz (see the computations made in \cite[Section 4]{imek-hold} in which $\widetilde{J}_{\frac{d}{2}-1}$ is denoted by $W$) and thus 
\begin{equation*}
n^{\frac{d}{2}-1}\Big|\widetilde{J}_{\frac{d}{2}-1}(\sqrt{2n+d}|x-y|)-\widetilde{J}_{\frac{d}{2}-1}(\sqrt{2n}|x-y|)\Big|=\mathcal{O}\Big( n^{\frac{d}{2}-1}(\sqrt{2n+d}-\sqrt{2n})\Big)=\mathcal{O}\big( n^{\frac{d-3}{2}}\big).
\end{equation*}
} as $n\rightarrow +\infty$
\begin{equation}\label{en-bessel}
e_{d,n}(x,y)=\frac{n^{\frac{d}{2}-1}}{(2\pi)^{\frac{d}{2}}} \widetilde{J}_{\frac{d}{2}-1}(\sqrt{2n}|x-y|)+\rem
\end{equation}
are uniform with respect to $(x,y)\in \S^{d-1} \times \S^{d-1}$ satisfying $|x-y|\leq 1$.

There exist constants $c>0$ and $\ep_d\in(0,1)$ satisfying the following asymptotic estimates uniformly with respect to $(x,y)$: 
\begin{eqnarray}\label{en-diag} 
e_{d,n}(x,x) &= &\frac{n^{\frac{d}{2}-1}}{(2\pi)^{\frac{d}{2}}\Gamma \left(\frac{d}{2} \right)} +\left\{ \begin{array}{lcl}\mathcal{O}\left(n^{\frac{d-3}{2}} \log(n)\right) &\mbox{for} & d\geq 3, \\
\mathcal{O}(n^{-\frac{1}{4}})&\mbox{for} & d= 2,\end{array}  \right.\\ \label{en-loin}
|e_{d,n}(x,y)|  & \leq & \ep_d e_{d,n}(x,x),\quad  \text{ if } \frac{c}{\sqrt{n}}\leq |x-y|\leq 1, \\
 \label{en-proche}
e_{d,n}(x,x)-e_{d,n}(x,y)  & \simeq &  n^{\frac{d}{2}}|x-y|^2, \quad \text{ if } |x-y|\leq \frac{c}{\sqrt{n}}.
\end{eqnarray}
The constants of the last equivalence may depend on $d$ only. 
\end{prop}

The asymptotics of Proposition \ref{en} give complements about the ones given in \cite[Proposition 4.1]{imek-crusep} and we believe that there should exist a proof via microlocal analysis, likely involving an analysis similar to that of \cite{zel-harmo} or maybe by an adaptation of \cite[line (8)]{canzani2015scaling} to the harmonic oscillator.
 However, we have found a proof using a completely different strategy: a specific use of the Mehler formula allows us to express $e_{d,n}(x,y)$ in a very simple way 
as a formula (see \eqref{edn}) involving the unidimensional Hermite functions. For instance, for $d=3$, our formula \eqref{edn} reads
\begin{equation*}
e_{3,n}(x,y)=\frac{1}{\pi}\sum_{0\leq \ell \leq \frac{n}{2}} h_{n-2\ell}\Big(\frac{|x+y|+|x-y|}{2} \Big)h_{n-2\ell}\Big(\frac{|x+y|-|x-y|}{2}\Big).
\end{equation*}

The proof of \eqref{en-bessel} is then a consequence of known approximations of unidimensional Hermite functions (see Proposition \ref{herm}) and an elementary use of the
Euler-Maclaurin formula on comparison of series and integrals.

Let us add that \eqref{en-diag} and \eqref{en-loin} are easy consequences of the asymptotics given by \eqref{en-bessel}. However \eqref{en-proche} needs a little work and cannot be deduced from \eqref{en-bessel}.

\begin{rema}
Due to a rotational invariance, the left-hand side of \eqref{en-diag} does not depend on $x\in \S^{d-1}$. More critical is the restriction $|x-y|\leq 1$ in \eqref{en-loin} since the estimates fail in general without this restriction, taking for instance $y=-x$ so that $|y-x|=2$ and using the following equality (which is a consequence of \eqref{edn} as we shall see below):
\begin{equation}\label{par}
  e_{d,n}(x,y)=(-1)^n e_{d,n}(x,-y).
\end{equation}
\end{rema}

Let us now write a few words about the organization of the paper and the notations we shall use.

As written above, Section \ref{proof-sp} contains a few graphical representations comforting the comparison \eqref{en-bessel} with the Bessel functions of Proposition \ref{en}.

Section \ref{proofteo} contains the proof of \cref{teo} upon admitting Proposition \ref{en}.
By comparison with the strategy used in \cite{imek-jep} that essentially deals with Dudley pseudo-distances involving term like $\min(1,n|x-y|)$, we remark in our paper that
in some sense all terms like $\min(1,n^\theta |x-y|)$ with $\theta>0$ play an equivalent role 
(see Proposition \ref{ups-the}). This allows us to deal with \eqref{equ2}.

Section \ref{proof-en} is devoted to the proof of Proposition \ref{en}. The starting point is to suitably compare the Mehler formula at dimensions $1$ and $d$ (see the proof of Proposition \ref{prop.edn}) in order to obtain a simple formula of $e_{d,n}(x,y)$.

Appendix \ref{anxA} gives elements of the proof of Theorem \ref{thm.main-harmonic}.

Appendices \ref{anxB} and \ref{pr-lemcos} finally contain a few technical lemmas.

The symbols $\simeq$, $\gtrsim$ and $\lesssim$ are always understood with constants of equivalence that may depend on the dimension $d$.

\section{Graphical representations of the spectral function restricted to $\S^{d-1}\times \S^{d-1}$}\label{proof-sp}

We recall that the spectral function associated to the harmonic oscillator $-\Delta+|x|^2$ on $\mathbb{R}^d$ acting on $E_n$, which is defined for any $(x,y)\in \R^d\times \R^d$ by 
\begin{equation*}
  e_{d,n}(x,y)=\sum_{k=1}^{\dim(E_n)}\varphi_{n,k}(x)\varphi_{n,k}(y)
\end{equation*}
in which $(\varphi_{n,k})$ is a Hilbert basis of $E_n$ (see \eqref{Enbasis}).
Although we are concerned with $d\geq 2$, it will be useful to consider the case $d=1$. The spectral function $e_{1,n}$ can indeed be written with $L^2(\R)$-normalized Hermite functions $h_n$ as  
\begin{equation}\label{fs}
  e_{1,n}(x,y)=h_n(x)h_n(y),\qquad \mbox{where} \quad h_n(x)=\frac{H_n(x)e^{-x^2/2}}{\sqrt{n! 2^n\sqrt{\pi} }},
\end{equation}
in which $(H_n)_{n\geq 0}$ is the standard sequence of Hermite polynomials ($H_0=1,H_1(x)=2x,\dots$). The following formula obtained by tensorization of $(h_n)$ is well-known:

\begin{equation*}
e_{d,n}(x,y)=\sum_{\substack{(i_1,\dots,i_d)\in \N^d \\ i_1+\dots+i_d=n}}h_{i_1}(x_1)h_{i_1}(y_1)\dots h_{i_d}(x_d)h_{i_d}(y_d).
\end{equation*} 

Since $-\Delta+|x|^2$ presents a rotational invariance, the spectral function $e_{d,n}$ should also have a similar invariance (that is not clear in the last formula). Actually, one can check that $e_{d,n}(x,y)$ merely depends on $|x-y|$ and $|x+y|$ (or equivalently $|x|^2+|y|^2$ and $\langle x,y\rangle$, see for instance the two Mehler formulas \eqref{mehler0} and \eqref{mehler}). We refer to Proposition \ref{prop.edn} for another formula of $e_{d,n}(x,y)$ that will play a fundamental role in the present paper.

In the statement of Proposition \ref{en}, due to the conditions $|x|=|y|=1$ and $|x-y|\leq 1$, we deduce that $e_{d,n}(x,y)$ merely depends\footnote{thanks to the identity $
|x+y|=\sqrt{4-|x-y|^2}$.} on $|x-y|$. Hence, we may write 
\begin{equation*}
e_{d,n}(x,y)=e_{d,n}\big(  (1,0,\dots,0),    (\cos(\alpha),\sin(\alpha),0,\dots,0) \big)
\end{equation*} 
for a suitable $\alpha\in \big[0,\frac{\pi}{3}\big]$ satisfying $|x-y|=|1-e^{i\alpha}|=2\sin(\alpha/2)$ and so $\alpha=2\arcsin\big(\frac{|x-y|}{2}\big)$.


The previous considerations allow us to obtain a few graphical representations for instance for $n=150$ with respect to $|x-y|$ running over $[0,1]$.

\begin{center}
\begin{tabular}{m{1cm} m{6cm} m{6cm}}
& $\qquad \qquad \qquad \frac{1}{n^{\frac{d}{2}-1}} e_{d,n}(x,y)$ & $\quad \quad \qquad \frac{1}{(2\pi)^{\frac{d}{2}}} \widetilde{J}_{\frac{d}{2}-1}(\sqrt{2n}|x-y|)$ \\
$d=4$ & \includegraphics[scale=0.3]{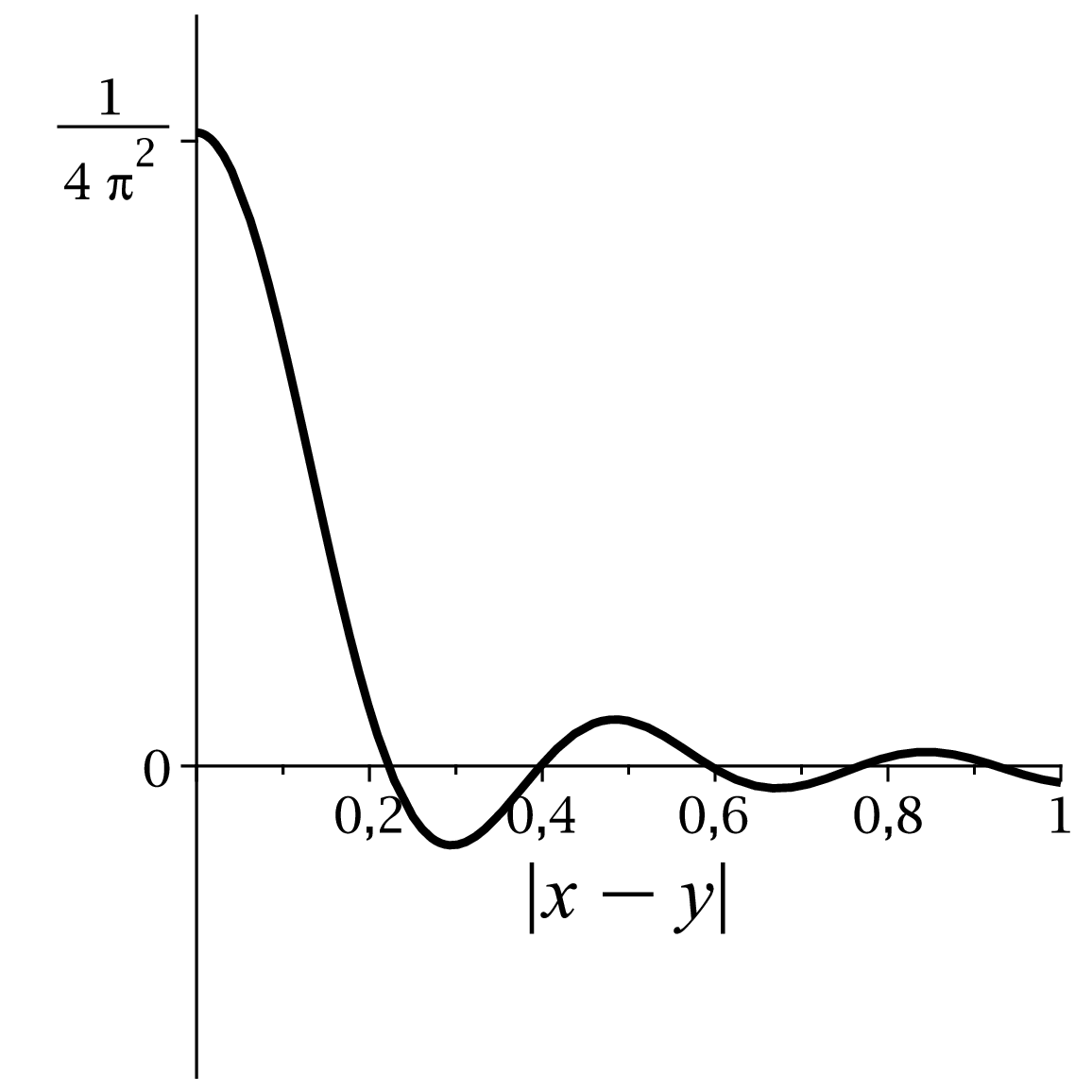} & \includegraphics[scale=0.3]{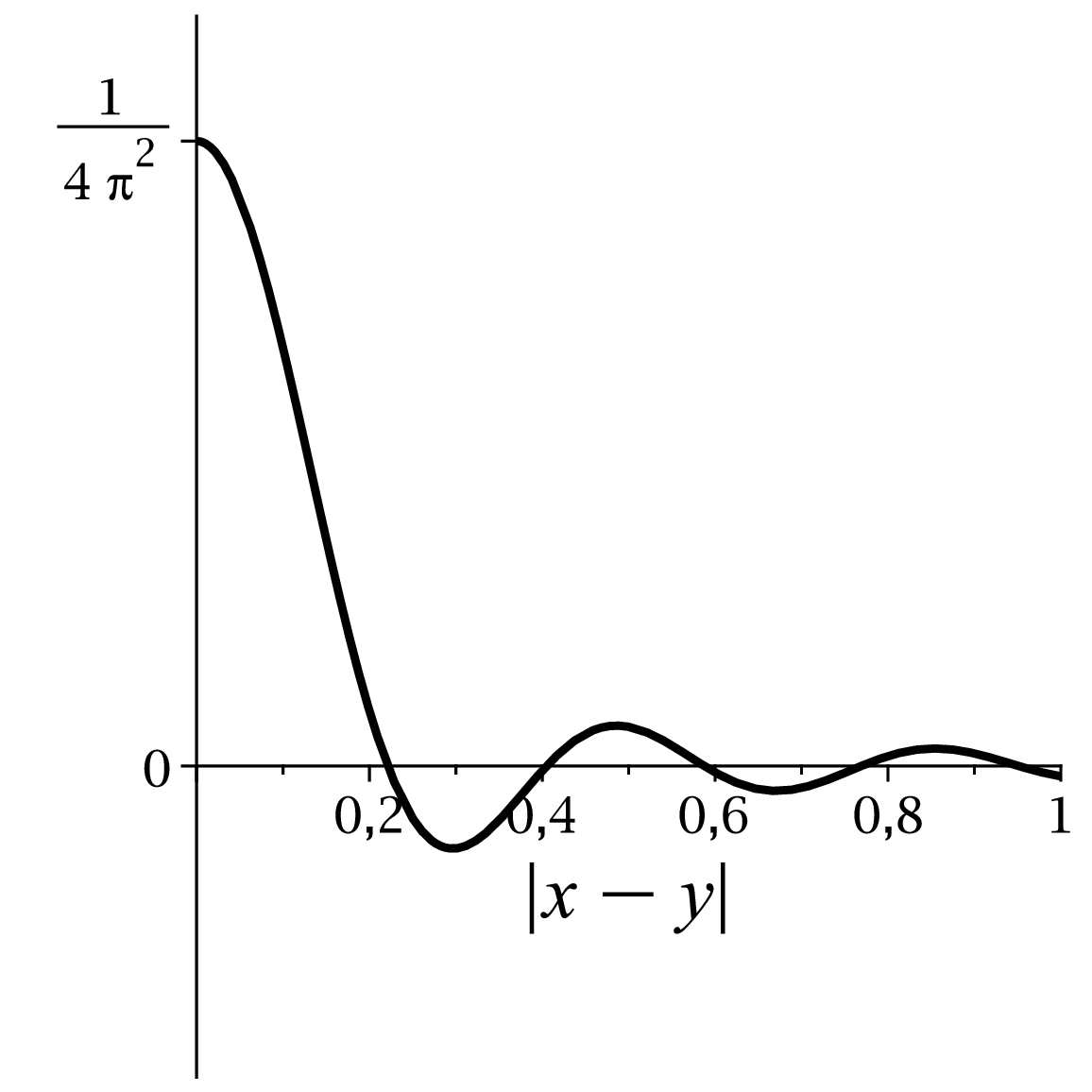}
\end{tabular}
\end{center}
\begin{center}
\begin{tabular}{m{1cm} m{6cm} m{6cm}}
$d=3$ & \includegraphics[scale=0.3]{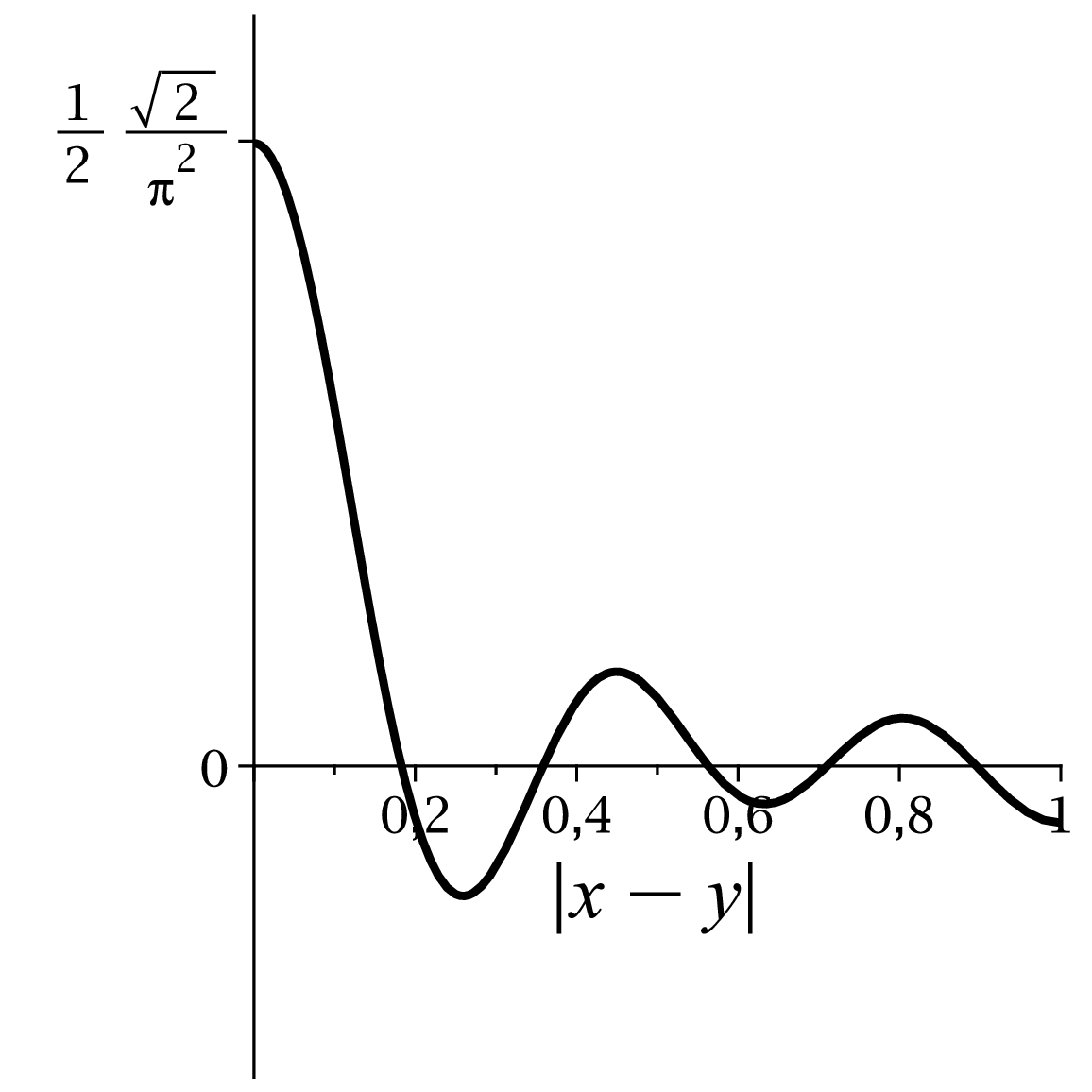} & \includegraphics[scale=0.3]{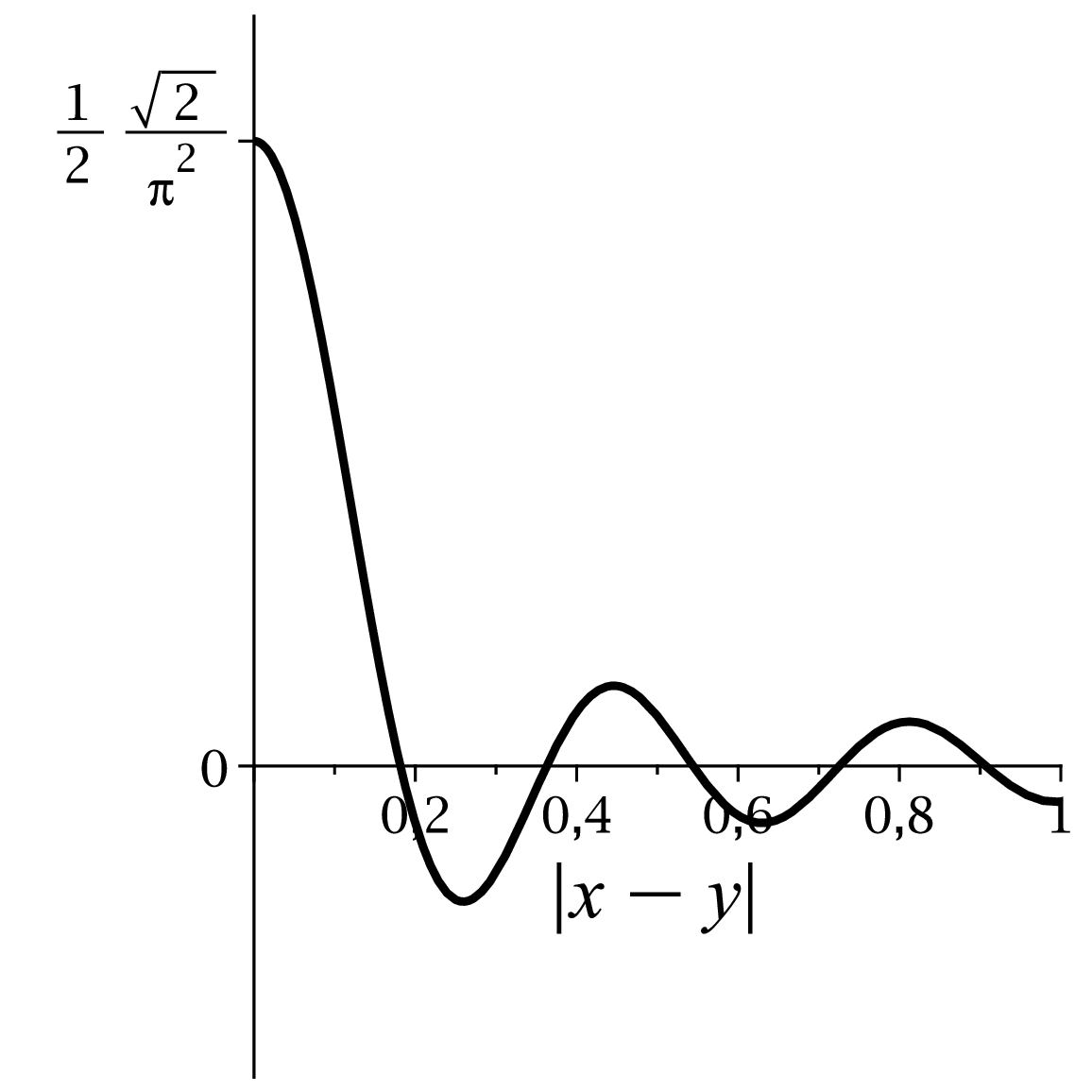}
\end{tabular}
\end{center}
\begin{center}
\begin{tabular}{m{1cm} m{6cm} m{6cm}}
$d=2$ & \includegraphics[scale=0.3]{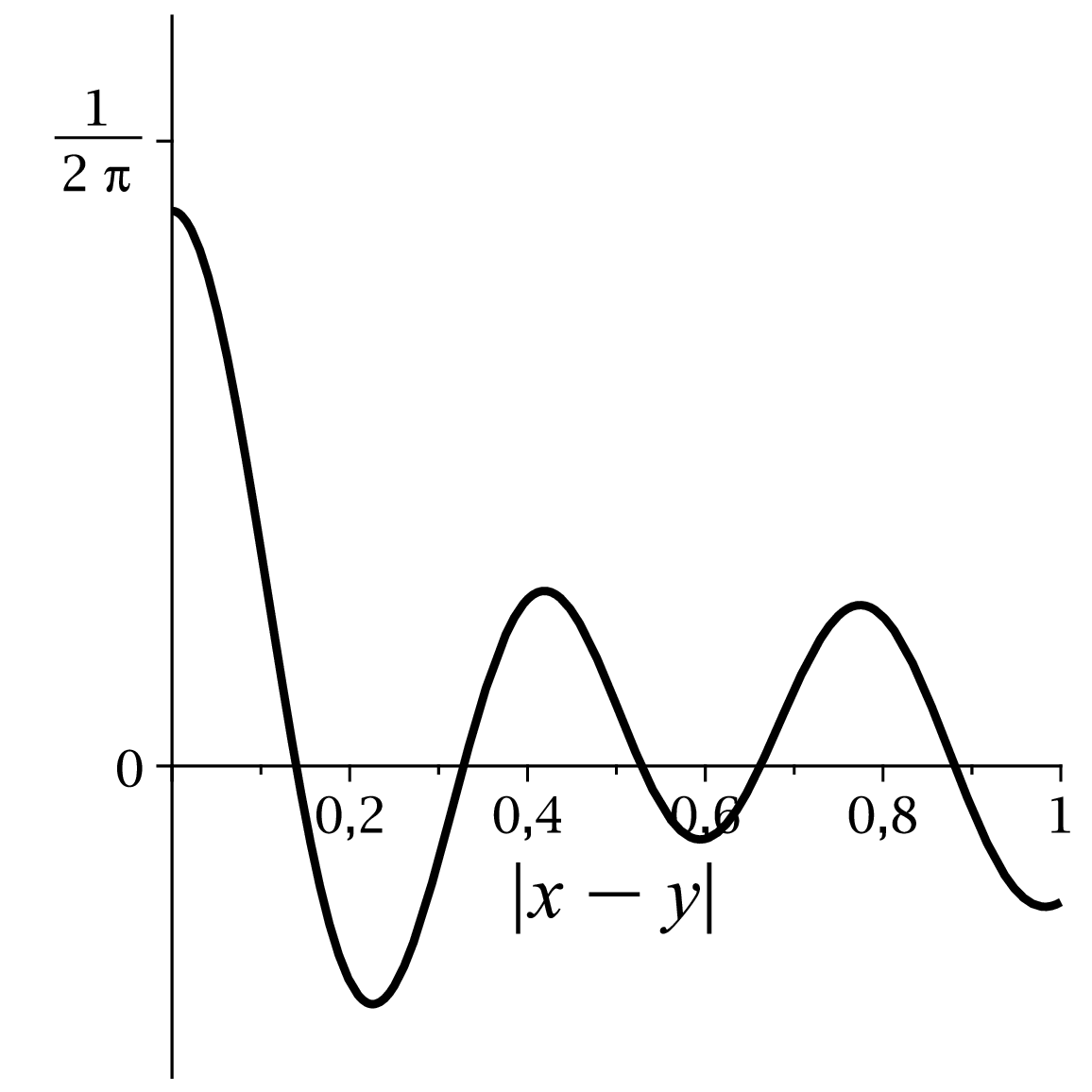} & \includegraphics[scale=0.3]{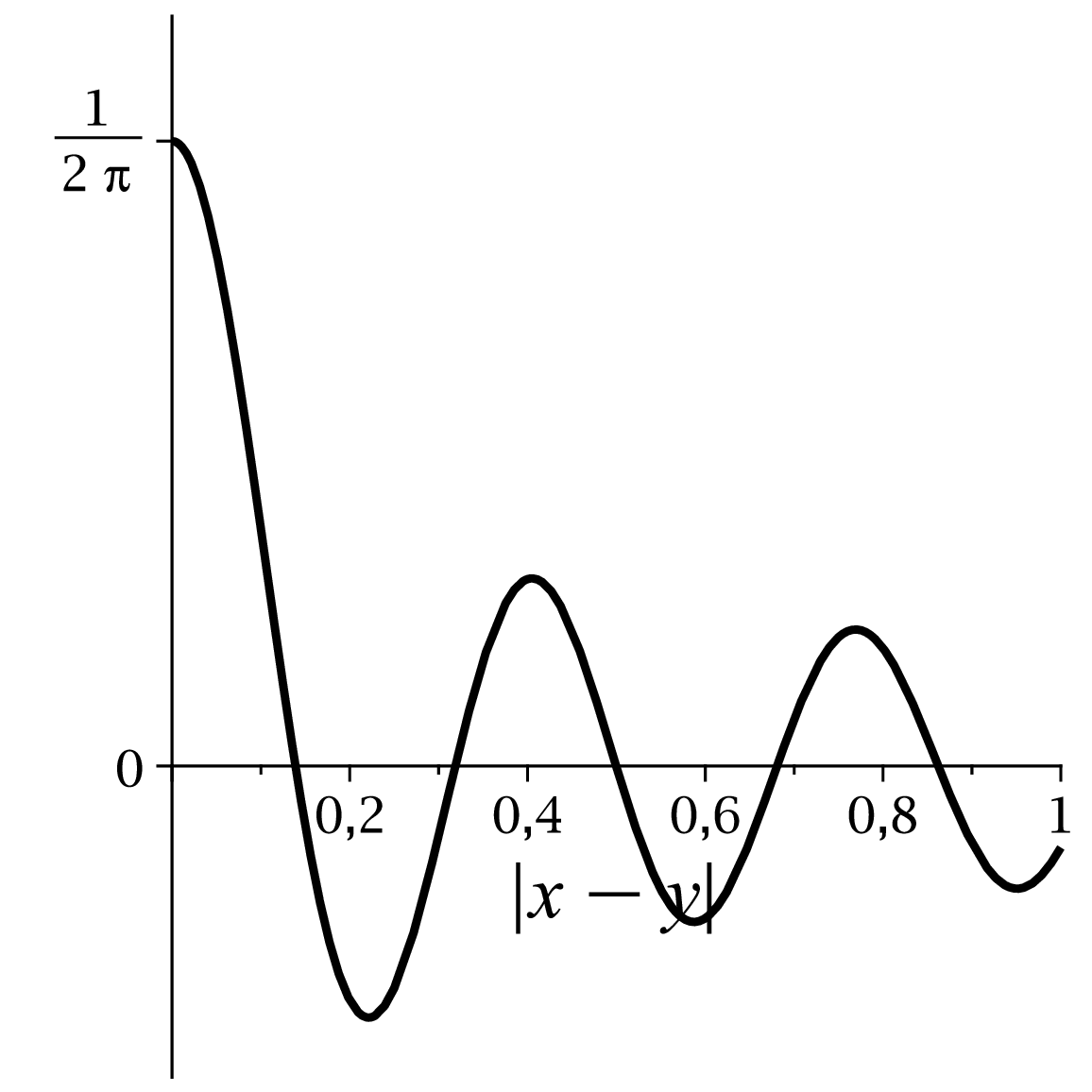}
\end{tabular}
\end{center}
In the last picture, the case $d=2$ seems to be slightly bad-approximated for $|x-y|=0$ and $n=150$. This is due to the slow remainder $n^{-\frac{1}{4}}$ in \eqref{en-diag}. Actually , we will give two proofs of \eqref{en-bessel} and the second one enlightens that the remainder $n^{-\frac{1}{4}}$ comes from the asymptotics of the Bessel function $J_0$.
More precisely, the formula \eqref{seco-term} below can be simplified as follows for $d=2$:
\begin{eqnarray}\label{red2a}
e_{2,n}(x,y)& =&\frac{1}{2\pi}\big( J_0(\sqrt{2n}|x-y|)+(-1)^n J_0(\sqrt{2n}|x+y|)\big)+
\mathcal{O}\Big(\frac{\log(n)}{\sqrt{n}}\Big),\\\label{red2b}
e_{2,n}(x,x)& =&\frac{1}{2\pi}\big( 1+(-1)^n J_0(2\sqrt{2n})\big)+
\mathcal{O}\Big(\frac{\log(n)}{\sqrt{n}}\Big).
\end{eqnarray}
We recall that $x$ and $y$ belong to $\S^{d-1}$ and that $e_{2,n}(x,x)$ does not depend on $x$.
Here is a graphical representation of the first points 
$(n,e_{2,n}(x,x))$ and we see that their ordinates are asymptotically 
\begin{equation*}
\frac{1}{2\pi}\big( 1\pm |J_0(2\sqrt{2n}|)\big)=\frac{1}{2\pi}+\mathcal{O}\Big(\frac{1}{n^{1/4}} \Big).
\end{equation*}

\begin{center}
\includegraphics[scale=0.75]{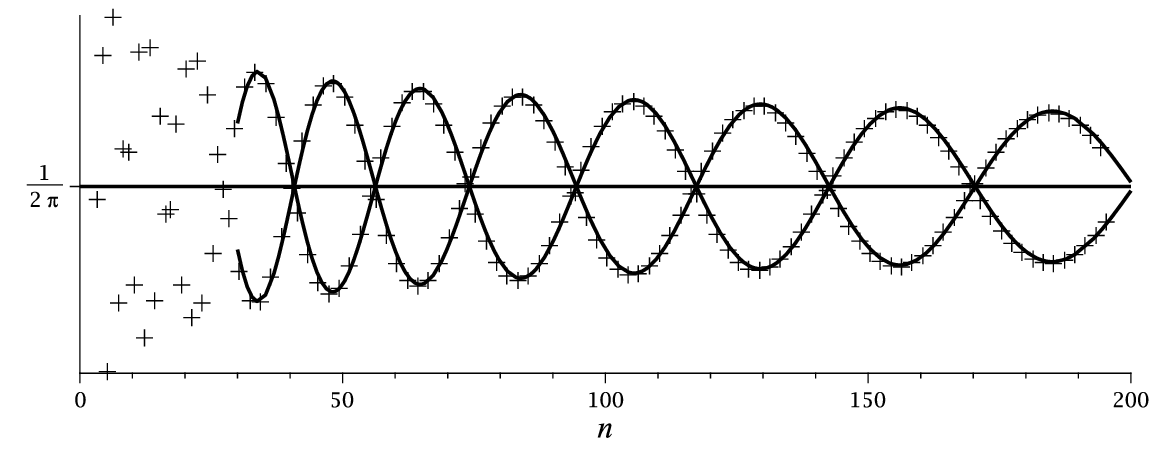}
\end{center}



\section{Proof of \cref{teo}}\label{proofteo}

As claimed in the introduction, the proof of the sufficient condition does not pose any major difficulty so we start with this part in \cref{sec.sufficient}. Then in \cref{sec.necessary} we give the proof of the necessary condition in \cref{teo} upon admitting crucial estimates on the spectral function whose proofs are delayed. 

\subsection{Proof of the sufficient condition}\label{sec.sufficient}

Our goal is to prove the implication (\textit{i}) $\Longrightarrow$ (\textit{ii}) of \cref{teo}.

We shall take our inspiration from the seminal paper \cite{salem1954} which uses the Cauchy condensation test with specific dyadic subseries and an adaptation of the argument of \cite[Section 6]{imek-aif} for Gaussian random variables. More precisely, the Cauchy condensation test allows to reformulate the condition \eqref{sz} as
\begin{equation*}
\sum_{\ell\geq 1} \frac{1}{\sqrt{\ell}} \Big(\sum_{n=2^\ell}^{+\infty} \frac{\|f_n\|_{L^2(\R^d)}^2}{n^{d/2}} \Big)^{1/2} <+\infty.
\end{equation*}
Another use of the Cauchy condensation test reformulates \eqref{sz} as
\begin{equation*}
\sum_{\ell\geq 0} 2^{\ell/2} \Big(\sum_{n=2^{2^\ell}}^{+\infty} \frac{\|f_n\|_{L^2(\R^d)}^2}{n^{d/2}}\Big)^{1/2} <+\infty,
\end{equation*}
so that in particular we get
\begin{equation}\label{sz2}
\sum_{\ell\geq 0} 2^{\ell/2} \Big(\sum_{2^{2^\ell}\leq n < 2^{2^{\ell+1}}} \frac{\|f_n\|_{L^2(\R^d)}^2}{n^{d/2}}\Big)^{1/2} <+\infty.
\end{equation}

We recall the following classical lemma. 

\begin{lemm}\label{lem.gaussian-sup} 
Let $(g_{i,j})_{i,j}$ be a collection of $I\times J$ i.i.d. standard $\mathcal{N}(0,1)$ Gaussian random variables. For any matrix $A\in \mathcal{M}_{I\times J}(\R)$ the following inequality holds true:
\begin{equation*}
\E_\omega\left[  \sup\limits_{1\leq i\leq I} \left\vert \sum_{j=1}^J a_{i,j} g_{i,j}(\omega) \right\vert\right]\leq C \sqrt{\log(2+I)} \sup\limits_{1\leq i \leq I} \Big(\sum_{j=1}^J a_{i,j}^2 \Big)^{1/2}.
\end{equation*}
\end{lemm}
\begin{proof}
See for instance of proof in \cite[Theorem IV.5, page 34]{queff1} for Rademacher random variables instead of Gaussian random variables (but the proof essentially relies on Khintchine inequalities which are known to be true in the Gaussian case, see \cite[Cor. V.27, page 256]{queff1}). See also \cite[Lemma 7.3, page 2746]{imek-aif}.
\end{proof}

\begin{proof}[Proof of Theorem \cref{teo} (i) $\Longrightarrow$ (ii)]
We start by applying a ``mesh strategy'' taken from \cite[Section 19]{imek-aif}: 
as seen in \cite{imek-aif}, for any integer $N\geq 1$, there exists a finite subset $\mathcal{X}_N$ of $\R^d$ satisfying $\log (2+\mbox{Card}(\mathcal{X}_N))\lesssim \log N$ and 
\begin{equation}\label{fini}
  \forall u\in  E_0+E_1+\dots+E_N\qquad \sup\limits_{x\in \R^d} |u(x)|\leq 2 \max\limits_{x\in\mathcal{X}_N} |u(x)|.
\end{equation}
Such a property is called the \textit{finite subset concentration} assumption in \cite{imek-aif} and comes from the Gaussian tail behavior of Hermite functions and from a Bernstein inequality.

Let us set $N=2^{2^{\ell+1}}$, so that we have $\sqrt{\log(2+\mbox{Card}(\mathcal{X}_N))}\lesssim 2^{\ell/2}$. Using \eqref{fini} and Lemma \ref{lem.gaussian-sup} we have: 
\begin{align*}
  \E_\omega\Big[\sup\limits_{x\in \R^d}  \Big\vert \sum_{2^{2^{\ell}}\leq n <2^{2^{\ell+1}}} f_n^{G,\omega}(x)   \Big\vert  \Big] &\leq 2\E_\omega\Big[\sup\limits_{x\in \mathcal{X}_{N}}  \Big\vert \sum_{2^{2^{\ell}}\leq n <2^{2^{\ell+1}}} f_n^{G,\omega}(x)   \Big\vert  \Big] \\
   & = 2 \E_\omega\Big[\sup\limits_{x\in \mathcal{X}_{N}}  \Big\vert \sum_{2^{2^{\ell}}\leq n <2^{2^{\ell+1}}} \frac{\|f_n\|_{L^2(\R^d)}}{\sqrt{\dim(E_n)}} \sum_{k=1}^{\dim(E_n)} g_{n,k}(\omega)\varphi_{n,k}(x)  \Big\vert  \Big] \\ 
  &\leq  C2^{\ell/2}\sup\limits_{x\in \mathcal{X}_N  } \Big(\sum_{2^{2^{\ell}}\leq n < 2^{2^{\ell+1}}} \frac{\|f_n\|_{L^2(\R^d)}^2}{\dim(E_n)} \sum_{k=1}^{\dim(E_n)} \varphi_{n,k}(x)^2   \Big)^{1/2}.
\end{align*}
Next, we use the equivalence $\dim(E_n)\sim c_dn^{d-1}$ and we recall that the spectral function $x\mapsto \sum\limits_{k=1}^{\dim(E_n)} \varphi_{n,k}(x)^2 $ associated to the harmonic oscillator satisfies the following uniform bound 
\begin{equation*}
  \sum_{k=1}^{\dim(E_n)} \varphi_{n,k}(x)^2 \lesssim n^{\frac{d}{2}-1},
\end{equation*}
for which we refer to \cite[Prop 4.1, page 330]{imek-crusep} or \cite[Coro 3.2 for $p=\infty$]{KoTa05}. Therefore, we get 
\begin{equation*}
  \E_\omega\Big[\sup\limits_{x\in \R^d}  \Big\vert \sum_{2^{2^{\ell}}\leq n <2^{2^{\ell+1}}} f_n^{G,\omega}(x) \Big\vert  \Big]  
  \leq C 2^{\ell/2} \Big(\sum_{2^{2^\ell}\leq n < 2^{2^{\ell+1}}} \frac{\|f_n\|_{L^2(\R^d)}^2}{n^{d/2}} \Big)^{1/2},
\end{equation*}
and from \eqref{sz2}, the following finiteness follows 
\begin{equation*}
\E_\omega\Big[\sum_{\ell\in \N} \sup\limits_{x\in \R^d}  \Big\vert \sum_{2^{2^{\ell}}\leq n <2^{2^{\ell+1}}} f_n^{G,\omega}(x) \Big\vert  \Big]  <+\infty.
 \end{equation*}
With probability $1$, one infers that the sequence of functions $\Big(\sum\limits_{n=0}^{2^{2^\ell}-1}  f_n^{G,\omega} \Big)_{\ell\in \N} $ uniformly converges in the Banach space $L^\infty(\R^d)$, which in turn implies that the Gaussian series $\sum f_n^{G,\omega}$ almost surely converges in $L^\infty(\R^d)$ thanks to 
\cite[Theorem III.5, page 132]{queff1}.
\end{proof}

\subsection{Proof of the necessary condition}\label{sec.necessary}

We prove the implication (\textit{iii}) $\Longrightarrow$ (\textit{i}) of \cref{teo}. As mentioned in the introduction this part heavily relies on the following estimate
(we recall that Proposition \ref{en} will be proved below). 

\begin{prop}\label{del-e} 
As $n\to \infty$, the following asymptotics of $\delta_n$, defined in \eqref{deln}, hold uniformly on the set $\{(x,y)\in \S^{d-1}\times \S^{d-1}: |x-y|\leq 1\}$,
\begin{equation*}
  \delta_n(x,y) \simeq n^{-\frac{d}{4}} \min(1,\sqrt{n}|x-y|).
\end{equation*}
\end{prop}

\begin{proof} In the following, we use the equality $e_{d,n}(x,x)=e_{d,n}(y,y)$, which is a consequence of the rotational invariance of the harmonic oscillator $-\Delta+|x|^2$. We write 
\begin{eqnarray*}
\delta_n(x,y)^2& =&   \frac{1}{\dim(E_n)} \Big(\sum_{k=1}^{\dim(E_n)}  \varphi_{n,k}(x)^2+\varphi_{n,k}(y)^2-2\varphi_{n,k}(x)\varphi_{n,k}(y) \Big)       \\
& =& \frac{1}{\dim(E_n)}\left(e_{d,n}(x,x)+e_{d,n}(y,y)-2e_{d,n}(x,y) \right)\\
& =& \frac{2}{\dim(E_n)}\left( e_{d,n}(x,x)-e_{d,n}(x,y) \right) \simeq \frac{1}{n^{d/2}}\min(1,\sqrt{n}|x-y|)^2,
\end{eqnarray*}
where in the last line the equivalence we have used \eqref{dimEn} and the two-sided estimates of $e_{d,n}(x,x)-e_{d,n}(x,y)$ given by \eqref{en-loin} and \eqref{en-proche}. 
\end{proof}

\begin{rema}
  Note that \eqref{par} shows that the conclusion of Proposition \ref{del-e} does not hold on the whole compact set $\S^{d-1}\times \S^{d-1}$ since we have $\delta_n(x,-x)=0$ for $n$ even.
\end{rema}

We now recall the following almost sure convergence result.

\begin{lemm}\label{lem.as-cv-Sd}  
Let us fix $x_0\in \S^{d-1}$. The following assertions are equivalent:
\begin{enumerate}[(i)]
  \item the Gaussian random series $\displaystyle\sum_{n\in\mathbb{N}} f_n^{G,\omega}(x_0)$ almost surely converges;
  \item the series $\displaystyle\sum_{n\in\mathbb{N}}\frac{\|f_n\|_{L^2(\R^d)}^2 }{n^{d/2}}$ converges.
\end{enumerate}
\end{lemm}

\begin{proof}
The proof can be found in \cite[Theorem 4.1, (1) $\Leftrightarrow$ (4)]{imek-aif}.
\end{proof}

The following result constructs a stationary Gaussian process that realizes a given Dudley pseudo-distance. 

\begin{prop}\label{cano} 
Let $(\lambda_n)_{n\in \N^\star}$ and $(c_n)_{n\in \N^\star}$ be non-negative sequences, and assume that $\sum\limits_{n\geq 1}c_n<+\infty$. 
There exists a stationary Gaussian process $(F^{G,\omega})$ on $\S^1$ that satisfies for any $(x,y)\in \S^1\times \S^1$:
\begin{equation*}
\E[|F^{G,\omega}(x)-F^{G,\omega}(y)|^2]^{\frac{1}{2}} \simeq \sqrt{\sum_{n\geq 1} c_n \min(1,\lambda_n |x-y|)^2}
\end{equation*}
in which $|x-y|$ is the Euclidean distance once $\S^1$ is embedded, as usual, in $\R^2$.

Moreover, for any given arc $\Gamma \subset \S^1$, the following three assertions are equivalent:
\begin{enumerate}[(i)]
\item the Gaussian process $(F^{G,\omega})$ admits a version which is sample-continuous on $\Gamma$;
\item the Gaussian process $(F^{G,\omega})$ admits a version which is sample-bounded on $\Gamma$;
\item the following condition is fulfilled 
\begin{equation}\label{ento}
\int_{0}^1  \sqrt{\sum\limits_{n\geq 1} c_n \min(1,\lambda_n t)^2}\frac{\mathrm{d}t}{t(-\log(t))^{\frac{1}{2}}} <+\infty.
\end{equation}
\end{enumerate}
\end{prop}

\begin{proof}
The construction of $(F^{G,\omega})$ is done in the sequel. Let us admit for a moment 
the existence of $(F^{G,\omega})$. Then we refer to \cite[Proposition 6]{imek-jep} in which it is shown that \eqref{ento} is equivalent the entropic condition of the Gaussian process $(F^{G,\omega})$, from which the equivalence between (\textit{i}), (\textit{ii}) and (\textit{iii}) is derived from the classical theorem of Dudley and Fernique \cite[Theorem 11.17 and 13.3]{ledoux}. In the same spirit, we refer to the proof of \cite[Proposition 16]{imek-jep}.

Our task is therefore reduced to the construction of $F^{G,\omega}$. First, we need to define a canonical Gaussian process on $\R^2$. To this end, let $U:\Omega\rightarrow \S^1$ be a random unit vector with uniform distribution. We then can define an isotropic stationary Gaussian process $F_n^{G,\omega}:\Omega\times \R^2\rightarrow \R$ such that its covariance operator satisfies that for all $(x,y)\in \R^2\times \R^2$,  
\begin{equation}\label{cao}
\E[F_n^{G,\omega}(x) F_n^{G,\omega}(y)]= \E\left[\cos(\langle U(\omega),\lambda_n(x-y) \rangle ) \right].
\end{equation}
Such a Gaussian process is well-defined (see for instance \cite[page 332]{estrade} and \cite[page 16]{azais}). By symmetry of the random variable $U$, we have
\begin{equation}\label{sin}
\E\left[\sin(\langle U(\omega),\lambda_n(x-y) \rangle ) \right]=0.
\end{equation}
Thanks to the formula of the Fourier transform of the canonical Riemannian measure of $\S^1$, \eqref{cao} and \eqref{sin} read: 
\begin{equation*}
\E[F_n^{G,\omega}(x) F_n^{G,\omega}(y)]=  \E\left[ e^{i\langle U(\omega),\lambda_n(x-y) \rangle} \right]=J_0(\lambda_n |x-y|),
\end{equation*}
where $J_0:z\mapsto \displaystyle\frac{1}{\pi}\displaystyle\int_0^{\pi}\exp(iz\cos\tau)\,\mathrm{d}\tau$ is the Bessel function.

We may assume that the Gaussian random variables $(F_n^{G,\omega}(x))_{n\in \N}$ are independent for any fixed $x\in \R^2$. Define for any $x\in \S^1$ the following Gaussian process: 
\begin{equation*}
F^{G,\omega}(x)=\sum_{n\geq 1} \sqrt{c_n} F_n^{G,\omega}(x).
\end{equation*}
Due to \eqref{cao}, we get $\E[F_n^{G,\omega}(x)^2]=1$ and hence the previous series converges in $L^2(\Omega)$. In other words, the Gaussian process $F^{G,\omega}$ is well-defined. Stationarity is straightforward, because the independence of the $(F_n^{G,\omega})_{n\geq 1}$ yields  
\begin{equation*}
  \E[F^{G,\omega}(x)F^{G,\omega}(y)]=\sum_{n\geq 1} c_n \E[F_n^{G,\omega}(x) F_n^{G,\omega}(y)]
\end{equation*}
which depends only on $x-y$. Finally, the Dudley pseudo-distance of the Gaussian process $(F^{G,\omega}(x))_{x\in \S^1}$ satisfies
\begin{eqnarray}\nonumber
\E[|F^{G,\omega}(x)-F^{G,\omega}(y)|^2] & =& \sum_{n\geq 1}2 c_n \left( 1-  J_0(\lambda_n |x-y|)\right) \\ \label{equ}
 & \simeq &  \sum_{n\geq 1} c_n \min(1,\lambda_n |x-y|)^2. 
\end{eqnarray}

\begin{center}
\begin{minipage}{7 cm}
The equivalence \eqref{equ} can be proved thanks to the following asymptotics (coming from the power series at $t=0$ of $J_0(t)$):
$$1-J_0(t) \underset{t\to 0}{\sim} \frac{t^2}{4} $$
and is indeed comforted by a graphical representation: 
$$ \forall t>0\quad 0.2 \leq \frac{1-J_0(t)}{\min(1,t)^2}\leq 1.5                   $$ 
\end{minipage}
\hspace*{0.5 cm}
\begin{minipage}{5 cm}
\includegraphics[scale=0.3]{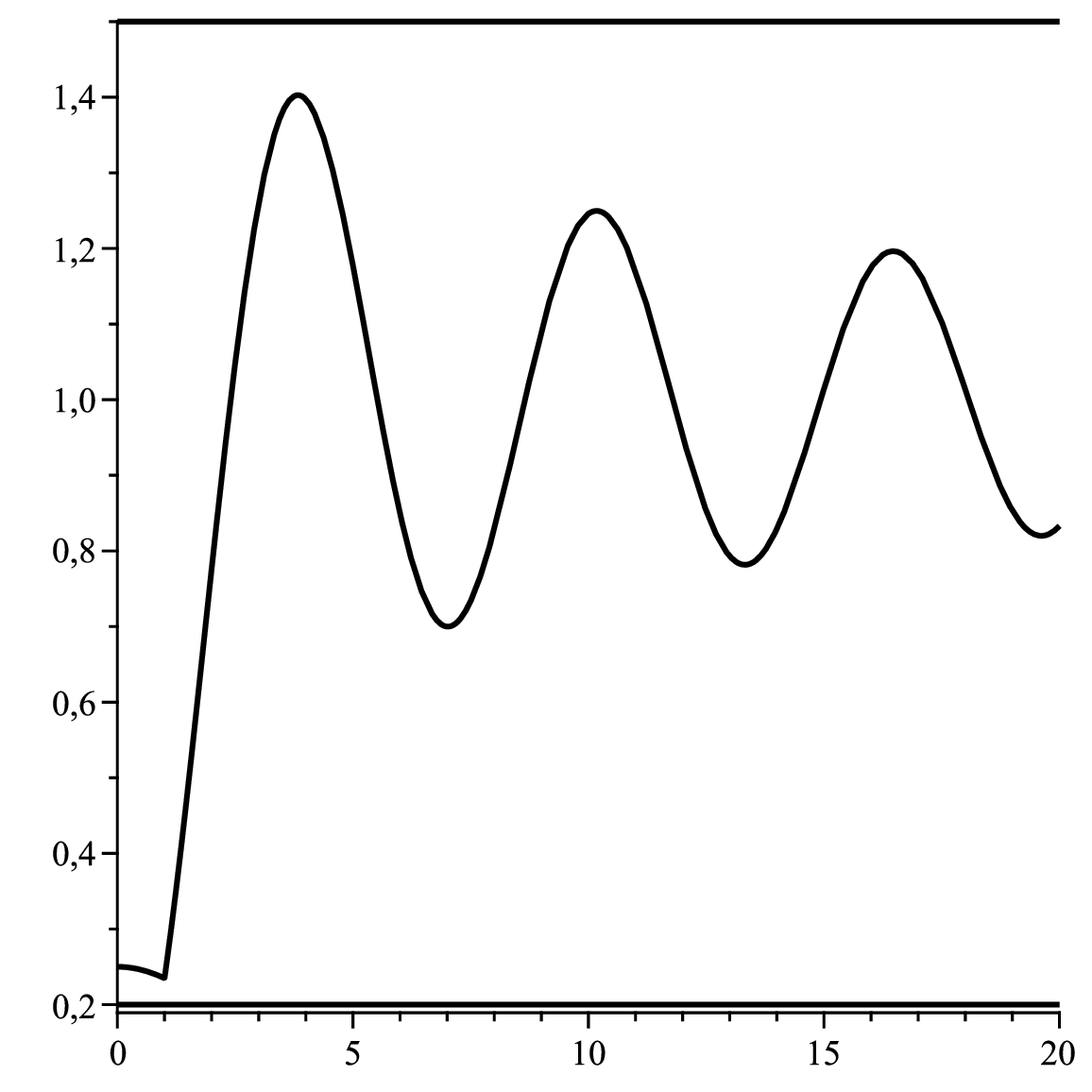}
\end{minipage}
\end{center}

\end{proof}

We may continue the proof of (iii) $\Longrightarrow$ (i) of \cref{teo}. Since we assume that (\textit{iii}) holds, we are in a position to apply Lemma \ref{lem.as-cv-Sd} and infer that $\sum \|f_n\|_{L^2(\mathbb{R}^d)}^2n^{-\frac{d}{2}}$ is a convergent series. 

 We consider the following curve on $\S^{d-1}$: 
\begin{equation*}
  \Gamma=\left\{(\cos(\alpha),\sin(\alpha),0,\dots,0)\in \R^d,\quad \alpha \in \left[0,\frac{\pi}{3}\right] \right\}.
\end{equation*}

Note that for any $(\gamma_1, \gamma_2)\in \Gamma\times \Gamma$ we have $|\gamma_1-\gamma_2|\leq 1$. 
Using \eqref{dedel} and Proposition \ref{del-e}, we see that the Dudley pseudo-distance $\delta$ of the Gaussian process $\left(f^{G,\omega}(\gamma) \right)_{\gamma \in \Gamma} $ satisfies for any $(\gamma_1,\gamma_2)\in \Gamma\times \Gamma$:

\begin{eqnarray*}
  \delta( \gamma_1,\gamma_2) = \E\left[ |f^{G,\omega}(\gamma_1)-f^{G,\omega}(\gamma_2)|^2\right]^{\frac{1}{2}} \simeq \sqrt{\sum_{n\geq 1} \frac{\|f_n\|_{L^2(\R^d)}^2}{n^{\frac{d}{2}}} \min(1,\sqrt{n}|\gamma_1-\gamma_2|)^2}.
\end{eqnarray*}
We then apply Proposition \ref{cano} with 
\begin{equation*}
c_n=\|f_n\|_{L^2(\R^d)}^2 n^{\frac{-d}{2}}\qquad \mbox{and}\qquad \lambda_n=\sqrt{n}
\end{equation*}
and we consider the associated Gaussian process $(F^{G,\omega})$ on $\S^1$ given by this result. 
We arrive at the equivalence 
\begin{equation*}
\E\left[|f^{G,\omega}(\gamma_1)-f^{G,\omega}(\gamma_2)|^2\right]^{\frac{1}{2}}\simeq
\E\left[|F^{G,\omega}(\gamma_1)-F^{G,\omega}(\gamma_2)|^2 \right]^{\frac{1}{2}}.
\end{equation*}
An application of the classical Gaussian comparison theorems of Slepian type ensures that the Gaussian process $(F^{G,\omega}(\gamma))_{\gamma \in \Gamma}$
admits a sample-bounded version (see \cite[Corollary 3.14]{ledoux}), because it is the case of the process $(f^{G,\omega}(\gamma))_{\gamma \in \Gamma}$. 

From Proposition \ref{cano}~ (\textit{iii}) we infer that, by setting $\theta=\frac{1}{2}$, the following holds 
\begin{equation*}
\int_{0}^1 \sqrt{\sum_{n\geq 1} \frac{\|f_n\|^2}{n^{\frac{d}{2}}} \min( 1, n^\theta t)^2} \frac{\mathrm{d}t}{t(-\log(t))^{\frac{1}{2}}}<+\infty, 
\end{equation*}
and the expected conclusion, namely assertion i) of Theorem \ref{teo}, is now a consequence of the following result.  

\begin{prop}\label{ups-the} 
  Let us consider $\theta\in (0,+\infty)$ and a nonnegative sequence $(c_n)_{n\geq 1}$ satisfying $\displaystyle\sum_{n\geq 1} c_n<\infty$. We introduce the function 
  \begin{equation*}
     \Upsilon_{\theta}(t)=\sqrt{\sum_{n\geq 1}  c_n \min(1,n^\theta t)^2} \quad \text{ for all } t\geq 0.
  \end{equation*}
  Then the following equivalence holds true
  \begin{equation}\label{szthe}
  \int_{0}^1 \Upsilon_\theta(t) \frac{\mathrm{d}t}{t(-\log(t))^{\frac{1}{2}}}  \simeq 
  \sum_{p\geq 1} 
  \frac{1}{p\sqrt{\log(p+1)}} \sqrt{\sum_{n\geq p } c_n}, 
  \end{equation}
  with constants that may depend on $\theta$ but the finiteness of \eqref{szthe} does not depend on $\theta$. 
  \end{prop}
  
The proof is essentially a completion of a computation made in \cite{imek-jep}.

\begin{proof} 
  We follow the proof of \cite[pages 781-782]{imek-jep}, the slight difference is in \eqref{sz22}. Let us first decompose: 
  \begin{align*}
  \int_{0}^1 \frac{\Upsilon_\theta(t)}{t (-\log(t))^{\frac{1}{2}}} \mathrm{d} t &= 
  \sum_{p\geq 1} \int_{1/(p+1)}^{1/p} \frac{\Upsilon_\theta(t)}{t (-\log(t))^{\frac{1}{2}}} \mathrm{d} t\\
  &\simeq \sum_{p\geq 1} \sqrt{\sum_{n\geq 1 }c_n \min \Big( 1,\frac{n^{2\theta}}{p^2} \Big)} \frac{1}{(\log(p+1))^{\frac{1}{2}}p}
  \end{align*}
in which we have controlled the non-decreasing function $\Upsilon_\theta$ for instance with the equivalence
\begin{equation*}
\int_{1/(p+1)}^{1/p} \frac{\mathrm{d}{t}}{t\sqrt{-\log(t)}} = 2\sqrt{\log(p+1)}-2\sqrt{\log(p)}\simeq \frac{1}{(\log(p+1))^{\frac{1}{2}} p}
\end{equation*}
and the following inequalities
\begin{equation*}
\frac{1}{2}\Upsilon_\theta\Big( \frac{1}{p}\Big)\leq  \Upsilon_\theta\Big( \frac{1}{2p}\Big)\leq \Upsilon_\theta\Big( \frac{1}{p+1}\Big)\leq \Upsilon_\theta(t) \leq \Upsilon_\theta\Big( \frac{1}{p}\Big)\qquad \forall t \in \Big[\frac{1}{p+1},\frac{1}{p} \Big].
\end{equation*}
  For any integer $p\geq 1$, we write
  \begin{equation}\label{Vpth}
    U_{p}(\theta)=\sum_{1\leq n < \lfloor p^{1/\theta}\rfloor}n^{2\theta} c_n \quad \text{and}\quad V_{p}(\theta)=\sum\limits_{n \geq  \lfloor p^{1/\theta}\rfloor}c_n.
  \end{equation}
 We remark that first integer $n^\star=\lfloor p^{\frac{1}{\theta}}\rfloor$ in $V_p(\theta)$ satisfies $p\leq (n^\star+1)^\theta$ and thus
\begin{equation*}
\frac{1}{2^{2\theta}}\leq  \frac{(n^\star)^{2\theta}}{p^2}\leq 1 
\end{equation*} 
  so that we can write (by separating the three cases $n<n^\star,n=n^\star$ and $n>n^\star$):
  \begin{eqnarray}\nonumber
 \sum\limits_{n\geq 1} c_n \min \Big(1,\frac{n^{2\theta}}{p^2} \Big) & \simeq & \frac{U_p(\theta)}{p^2}+V_p(\theta)\\\nonumber
  \int_{0}^1 \frac{\Upsilon_\theta(t)}{t(-\log(t))^{\frac{1}{2}}} \mathrm{d} t &\simeq &
  \sum_{p\geq 1} \sqrt{\frac{1}{p^2}U_{p}(\theta)+V_{p}(\theta)} \frac{1}{p\sqrt{\log(p+1)}} \\ \label{sz0}
   & \simeq & \sum_{p\geq 1} \frac{\sqrt{U_{p}(\theta)}}{p^2 \sqrt{\log(p+1)}}+\sum_{p\geq 1} \frac{\sqrt{V_{p}(\theta)}}{p\sqrt{\log(p+1)}}.
  \end{eqnarray}
  In contrast to \cite[line (75)]{imek-jep}, we have to show
  \begin{equation}\label{sz1}
  \sum_{p\geq 1} \frac{\sqrt{U_{p}(\theta)}}{p^2 \sqrt{\log(p+1)}}\lesssim \sum_{p\geq 1} \frac{\sqrt{V_p(1)}}{p\sqrt{\log(p+1)}}
  \end{equation}
  and also the independence with respect to $\theta$:
  \begin{equation}\label{sz22}
  \sum_{p\geq 1} \frac{\sqrt{V_{p}(\theta)}}{p\sqrt{\log(p+1)}}\simeq 
  \sum_{p\geq 1} \frac{\sqrt{V_p(1)}}{p\sqrt{\log(p+1)}}
  \end{equation}
  Clearly \eqref{sz0}, \eqref{sz1} and \eqref{sz22} together imply \eqref{szthe} since the $V_p(1)$ part is exactly the Salem-Zygmund term appearing in the right-hand side of \eqref{szthe}.
  
  \textbf{Proof of \eqref{sz1}.}
  By writing $\frac{\sqrt{U_{p}(\theta)}}{p^2 \sqrt{\log(p+1)}}=\frac{1}{\sqrt{p}\log(p+1)}\times \frac{\sqrt{U_{p}(\theta)} \sqrt{\log(p+1)}}{p^{3/2}}$, we take advantage of the argument of \cite[page 782]{imek-jep} by using the Cauchy-Schwarz inequality to get
  \begin{eqnarray*}
  \sum_{p\geq 1} \frac{\sqrt{U_{p}(\theta)}}{p^2 \sqrt{\log(p+1)}} & \leq & C \Big(\sum_{p\geq 1} \frac{U_{p}(\theta) \log(p+1)}{p^3}       \Big)^{\frac{1}{2}}\\
  & \leq & C\Big(    \sum_{n\geq 1} c_n n^{2\theta} \sum_{p\geq \lceil n^\theta \rceil } \frac{\log(p+1)}{p^3}          \Big)^{\frac{1}{2}}
  \end{eqnarray*}
which is incidentally enough for our purpose, as one ultimately obtains 
  \begin{equation*}
    \sum_{p\geq 1} \frac{\sqrt{U_{p}(\theta)}}{p^2 \sqrt{\log(p+1)}}  \leq C \Big(  \sum_{n\geq 1}  c_n \frac{n^{2\theta} \log(\lceil n^\theta \rceil+1)}{\lceil n^\theta \rceil^2}        \Big)^{\frac{1}{2}} \simeq 
  \Big(  \sum_{n\geq 1}  c_n \log(n+1)  \Big)^{\frac{1}{2}}, 
  \end{equation*}
  giving fortunately the same term as in \cite[page 782]{imek-jep} whose argument thus may be continued to get \eqref{sz1}.
 
  \textbf{Proof of \eqref{sz22} for $\theta\in (0,1)$.} It is important to keep in mind that $\theta\mapsto V_{p}(\theta)$ is a non-decreasing function. 
  
  From this observation, it follows that $V_{p}(\theta)\leq V_p(1)$ for any $\theta \in (0,1)$, which readily shows the upper bound of \eqref{sz22} in the case $\theta \in (0,1)$. 
  
  The key observation which allows to deal with the remaining cases, and whose proof is postponed, is that for any integer $T\geq 1$ the following holds: 
  \begin{equation}\label{VPT}
  \sum_{p\geq 1} \frac{\sqrt{V_{p}(T)}}{p\sqrt{\log(p+1)}} \lesssim 
  \sum_{p\geq 1} \frac{\sqrt{V_{p}(1/T)}}{p\sqrt{\log(p+1)}}.
  \end{equation} 
  
  Let us now show the lower bound in \eqref{sz22} in the case $\theta \in (0,1)$. 
  We first choose an integer $T$ which satisfies $\frac{1}{T}< \theta$. Therefore, we obtain $V_p(1)\leq V_{p}(T)$ and $V_{p}(1/T)\leq V_{p}(\theta)$, so that \eqref{sz22} is a consequence of \eqref{VPT}. 
  
  \textbf{Proof of \eqref{sz22} for $\theta\geq 1$.} As above, the lower-bound is a direct consequence of the inequality $V_p(1)\leq V_p(\theta)$. It remains to address the upper bound of \eqref{sz22}. We consider an integer $T$ satisfying $T>\theta$ and therefore we have $V_{p}(\theta)\leq V_{p}(T)$ and $V_{p}(1/T)\leq V_p(1)$, and we conclude as above thanks to \eqref{VPT}.
  
\textbf{Proof of \eqref{VPT}} We observe that since the sequences inside the two series in \eqref{VPT} are non-increasing with respect to $p$, the Cauchy condensation test shows that \eqref{VPT} is indeed equivalent to 
  \begin{equation*}
  \sum_{p\geq 0} \frac{\sqrt{V_{2^p}(T)}}{\sqrt{p+1}} \lesssim 
  \sum_{p\geq 0} \frac{\sqrt{V_{2^p}(1/T)}}{\sqrt{p+1}}.
  \end{equation*}
  
  We now split the left-hand side as a sum modulo $T^2$ in the following way: 
  \begin{equation*}
  \sum_{p\geq 0} \frac{\sqrt{V_{2^p}(T)}}{\sqrt{p+1}}=\sum_{k=0}^{T^2-1} 
  \sum_{p\geq 0} \frac{\sqrt{V_{2^{pT^2+k}}(T)}}{\sqrt{pT^2+k+1}}
  \leq 
  \sum_{p\geq 0}\sqrt{V_{2^{pT^2}}(T)} \Big(\sum_{k=0}^{T^2-1} \frac{1}{\sqrt{pT^2+k+1}} \Big)
  \end{equation*}
  which can be simplified (absorbing a factor $T$ in the constant): 
  \begin{equation*}
  \sum_{p\geq 0} \frac{\sqrt{V_{2^p}(T)}}{\sqrt{p+1}}\lesssim \sum_{p\geq 0}\frac{\sqrt{V_{2^{pT^2}}(T)} }{\sqrt{p+1}}.
  \end{equation*}
  The conclusion follows from the equality $V_{2^{pT^2}}(T)=V_{2^{p}}(1/T)$ that can be seen directly from the definition \eqref{Vpth}. 
\end{proof}

\section{Proof of Proposition \ref{en}}\label{proof-en}

\subsection{Preliminaries}
We give here some results needed in our proof of Proposition \ref{en}.
Our starting point is the Mehler formula, which holds for any $d\geq 1$, $(x,y)\in \R^d\times \R^d$ and any $t>0$, 
\begin{equation}\label{mehler0}
\sum_{n=0}^{+\infty} e^{-t(2n+d)}e_{d,n}(x,y)=\frac{1}{(2\pi \sinh(2t))^{\frac{d}{2}}}\exp\left(\frac{-\tanh(t)}{4}|x+y|^2-\frac{|x-y|^2}{4\tanh(t)} \right).
\end{equation}
The exponent in the right-hand side is also often written as follows
\begin{equation}\label{mehler}
\sum_{n=0}^{+\infty} e^{-t(2n+d)}e_{d,n}(x,y)=\frac{1}{(2\pi \sinh(2t))^{\frac{d}{2}}}\exp\left( -\frac{|x|^2+|y|^2}{2\tanh(2t)}+\frac{\langle x,y\rangle}{\sinh(2t)} \right).
\end{equation}
By tensorization, the previous formula is a consequence of the case $d=1$ (see \cite[page 380]{szeg}).

A crucial remark is that the exponent in the Mehler formula merely depends on $|x|^2+|y|^2$ and $\langle x,y\rangle$. Next, we change variables by introducing
\begin{equation}\label{hat}
\hat{x}=\frac{|x+y|+|x-y|}{2} \quad \mbox{and}\quad\hat{y}=\frac{|x+y|-|x-y|}{2}
\end{equation}
for which it is easy to check that $(x,y)\mapsto (\hat{x},\hat{y})$ is a non-linear isometry:
\begin{equation}\label{hat2}
\hat{x}^2+\hat{y}^2=|x|^2+|y|^2, \qquad  \hat{x}\hat{y}=\langle x,y\rangle, \qquad \hat{x}-\hat{y}=|x-y|.
\end{equation}
With these new variables, we may state the following result which expresses $e_{d,n}(x,y)$ as a Cauchy product. 

\begin{prop}\label{prop.edn} 
Let us consider $d\geq 2$, $n\in \N$, and $(x,y)\in \R^d\times \R^d$. Then the following equality holds: 
\begin{equation}\label{edn}
e_{d,n}(x,y)=\frac{1}{\pi^{\frac{d-1}{2}}}\sum_{\substack{(k,\ell)\in \N\times \N \\ k+2\ell=n}} h_k(\hat{x})h_k(\hat{y}) B_d(\ell),
\end{equation}
where $B_d(\ell)$ satisfies the following asymptotics for $\ell\gg 1$:
\begin{equation}\label{Adell}
B_d(\ell)=\frac{\Gamma\left(\frac{d-1}{2}+\ell \right)}{\ell! \Gamma \left(\frac{d-1}{2} \right)}=\frac{1}{\Gamma(\frac{d-1}{2})} \ell^{\frac{d-3}{2}}+\mathcal{O}\left(\ell^{\frac{d-5}{2}}\right).
\end{equation}
\end{prop}

\begin{proof} 
We start by observing that thanks to \eqref{hat2} we have 
\begin{equation*}
  \exp\left(-\frac{|x|^2+|y|^2}{2\tanh(2t)}+\frac{\langle x,y\rangle}{\sinh(2t)}\right) = \exp\left(
   -\frac{\hat{x}^2+\hat{y}^2}{2\tanh(2t)}+\frac{\hat{x}\hat{y}}{\sinh(2t)} \right).
\end{equation*}
Since $\hat{x}$ and $\hat{y}$ are real, we observe that the Mehler formula \eqref{mehler} in dimension $d$ evaluated at $(x,y)$ can be rewritten using the Mehler formula in dimension $1$ evaluated at $(\hat{x}, \hat{y})$. We have indeed  
\begin{eqnarray*}
(2\pi \sinh(2t))^{\frac{d}{2}}\sum_{n=0}^{+\infty} e^{-t(2n+d)}e_{d,n}(x,y) & =&  (2\pi \sinh(2t))^{\frac{1}{2}}\sum_{k=0}^{+\infty} e^{-t(2k+1)}e_{1,k}(\hat{x},\hat{y})\\
& =& (2\pi \sinh(2t))^{\frac{1}{2}}\sum_{k=0}^{+\infty} e^{-t(2k+1)} h_k(\hat{x}) h_k(\hat{y}).
\end{eqnarray*}
We are now in a position to study the behavior as $t\rightarrow +\infty$. To this end we write $t=-\log(z)$, which corresponds to the limit $z\rightarrow 0^{+}$. Using the formula $-2\sinh(2\log (z))=z^{-2}(1-z^4)$, we get 
\begin{eqnarray*}
\sum_{n=0}^{+\infty} z^{2n+d} e_{d,n}(x,y) & =&  \frac{z^{d-1}}{\pi^{\frac{d-1}{2}} (1-z^4)^{\frac{d-1}{2}}}\sum_{k=0}^{+\infty}z^{2k+1} h_k(\hat{x})h_k(\hat{y}) \\
& = & \frac{1}{\pi^{\frac{d-1}{2}}}  (1-z^4)^{\frac{-(d-1)}{2}}     \sum_{k=0}^{+\infty} z^{2k+d} h_k(\hat{x})h_k(\hat{y}) \\
& =& \frac{1}{\pi^{\frac{d-1}{2}}}\Big( \sum_{\ell=0}^{+\infty} B_d(\ell)z^{4\ell} \Big)\Big( \sum_{k=0}^{+\infty} z^{2k+d} h_k(\hat{x})h_k(\hat{y})\Big),
\end{eqnarray*}
in which we have denoted by $B_d(\ell)=\comb{\frac{-(d-1)}{2}}{\ell}(-1)^\ell$. It follows that \eqref{edn} holds because 
\[
  B_d(\ell) = \frac{\left(\frac{d-1}{2}\right)\left(\frac{d-1}{2}+1\right)\dots \left(\frac{d-1}{2}+\ell-1\right)}{\ell!} = \frac{\Gamma\left( \frac{d-1}{2}+\ell\right)}{\ell!\Gamma\left(\frac{d-1}{2}  \right)}.
\]
In order to obtain \eqref{Adell}, one can use the well-known asymptotics as $t\rightarrow +\infty$:
\begin{equation*}
\log \Gamma(t)=t\log(t)-t-\frac{1}{2}\log(t)+\log\sqrt{2\pi}+\mathcal{O}\left( \frac{1}{t}\right)
\end{equation*}
so that for any real number $a$ we get for $t\gg 1$: 
\begin{equation*}
\log \Gamma(t+a)-\log \Gamma(t)=
a\log(t)+\mathcal{O}\left( \frac{1}{t}\right).
\end{equation*}
By subtracting the cases $a=\frac{d-1}{2}$ and $a=1$, we finally obtain

\begin{eqnarray*}
  B_d(\ell)& =&  \frac{1}{\Gamma\left(\frac{d-1}{2}  \right)}\exp\left( \log \Gamma \left( \ell+\frac{d-1}{2}\right)-\log \Gamma(\ell+1)  \right) \\
   &=& \frac{1}{\Gamma\left(\frac{d-1}{2}\right)} \exp\left(  \frac{d-3}{2} \log(\ell)+\mathcal{O}\left( \frac{1}{\ell}\right)\right)= \frac{\ell^{\frac{d-3}{2}}}{\Gamma\left(\frac{d-1}{2}  \right)}+\mathcal{O}\left(\ell^{\frac{d-5}{2}} \right).\qedhere
\end{eqnarray*}
\end{proof}
We now recall known uniform asymptotics of the Hermite functions $h_n$ on a compact interval (in the so-called allowed region) obtained by Muckenhoupt.

\begin{prop}\label{herm} 
As $n\rightarrow +\infty$, the following asymptotics hold uniformly for $s$ and $t$ each belonging to a same compact interval of $\R$,
\begin{eqnarray}\label{hn1}
  h_n(s)  &= & \frac{2^{1/4}}{\sqrt{\pi} n^{1/4}} \cos\left(s\sqrt{2n}-\frac{n\pi}{2}\right) +\mathcal{O} \left(\frac{1}{n^{3/4}} \right), \\ 
\label{hn3}
h_{n-1}(s)^2+h_{n}(s)^2 &=& \frac{\sqrt{2}}{\pi \sqrt{n}}+\mathcal{O}\left( \frac{1}{n}\right),
\end{eqnarray}
\begin{equation}
\label{hn2}
  h_n(s)h_n(t) =   \frac{1}{\pi\sqrt{2n}} \left[\cos\left((s-t)\sqrt{2n}\right)+(-1)^n\cos\left((s+t)\sqrt{2n}\right) \right]  +\mathcal{O}\left( \frac{1}{n}\right).\\
  \end{equation}
\end{prop}

\begin{proof} 
We see that \eqref{hn1} is a consequence of known asymptotics of Hermite functions (see \cite[eq.\ (2.5)]{muckenI} which read
\begin{equation*}
h_n(s)= \frac{\sqrt{2}}{\sqrt{\pi}(2n+1)^{\frac{1}{4}}} \cos\left(s\sqrt{2n+1}-\frac{n\pi}{2}\right)+\mathcal{O}\left( \frac{1}{n^{3/4}} \right).
\end{equation*}
The asymptotics \eqref{hn2} and \eqref{hn3} are then direct consequences of \eqref{hn1}.
\end{proof}

It is then possible to isolate the main contribution in \eqref{edn} as follows. 

\begin{prop}\label{unif-rem}
Let us consider $d\geq 2$ and $(x, y)\in \R^{d}\times \R^{d}$. The spectral function $e_{d,n}(x,y)$ satisfies the following for $n\gg 1$:
\begin{eqnarray}\label{en-pr}
e_{d,n}(x,y) &=& \frac{1}{\pi^{\frac{d+1}{2} }\sqrt{2} \Gamma \left( \frac{d-1}{2}\right)} \sum_{\substack{(k,\ell)\in \N^\star \times \N^\star \\ 
k+2\ell=n}} \frac{\cos((\hat{x}-\hat{y})\sqrt{2k})  }{\sqrt{k}}  \ell^{\frac{d-3}{2}} \\ \label{en-pr2}
& & +\frac{(-1)^n}{\pi^{\frac{d+1}{2} }\sqrt{2} \Gamma \left( \frac{d-1}{2}\right)} \sum_{\substack{(k,\ell)\in \N^\star \times \N^\star \\ 
k+2\ell=n}} \frac{  \cos( (\hat{x}+\hat{y})\sqrt{2k} )}{\sqrt{k}}  \ell^{\frac{d-3}{2}} \\
& & + \mathcal{O}\left(\log(n) n^{\frac{d-3}{2}} \right) \notag 
\end{eqnarray}
where the remainder is uniform with respect to $|x|+|y|$.
\end{prop}

\begin{proof} 
Our starting point is the exact formula \eqref{edn} which writes 
\begin{equation*}
e_{d,n}(x,y) = \frac{1}{\pi^{\frac{d-1}{2}} \Gamma \left( \frac{d-1}{2}\right)} \sum_{\substack{(k,\ell)\in \N\times \N \\
k+2\ell=n}} h_k(\hat{x})h_k(\hat{y})\left( \ell^{\frac{d-3}{2}}+\mathcal{O}(\ell^{\frac{d-5}{2}})\right).
\end{equation*}
Let us set $\mathcal{R}=|x|+|y|$.
Note that \eqref{hat} implies the bounds $|\hat{y}|\leq \hat{x}\leq \mathcal{R}$. Thanks to Proposition \ref{herm} we have $\|h_k\|_{L^\infty(-\mathcal{R},\mathcal{R})}\leq \frac{C}{(1+k)^{\frac{1}{4}}}$ for any $k\in \N$ and therefore we can discard the eventual endpoints $k=0$ and $k=n$ so that 
\begin{equation*}
e_{d,n}(x,y) = \frac{1}{\pi^{\frac{d-1}{2}} \Gamma \left( \frac{d-1}{2}\right)} \sum_{\substack{(k,\ell)\in \N^\star \times \N^\star \\
k+2\ell=n}} h_k(\hat{x})h_k(\hat{y}) \left( \ell^{\frac{d-3}{2}}+\mathcal{O}(\ell^{\frac{d-5}{2}}) \right)+\mathcal{O}\Big( n^{\frac{d-3}{2}}+\frac{1}{\sqrt{n}} \Big).
\end{equation*}
Since $d\geq 2$, the remainder is indeed $\mathcal{O}\big( n^{\frac{d-3}{2}}\big)$. An application of Proposition \ref{herm} shows that 
\begin{equation*}
h_k(\hat{x})h_k(\hat{y}) \left( \ell^{\frac{d-3}{2}}+\mathcal{O}(\ell^{\frac{d-5}{2}}) \right)
\end{equation*}
equals
\begin{equation}\label{remrem}
\frac{1}{\pi \sqrt{2k}}\left[ \cos( (\hat{x}-\hat{y})\sqrt{2k})+(-1)^k  \cos( (\hat{x}+\hat{y})\sqrt{2k} )  \right]  \ell^{\frac{d-3}{2}}  +\mathcal{O}\Big( \frac{\ell^{\frac{d-5}{2}}}{\sqrt{k}}+ \frac{\ell^{\frac{d-3}{2}}}{k}\Big)
\end{equation}
and the conclusion will be obtained by summation of the remainders and the obvious remark $(-1)^{k}=(-1)^{n-2\ell}=(-1)^n$. For the remainders, this indeed follows by separating the contributions $k\leq \frac{n}{2}$ and $k>\frac{n}{2}$ which gives the three estimates 
\begin{eqnarray*}
d=2 &\Rightarrow &   \sum_{\substack{(k,\ell)\in \N^\star \times \N^\star \\
  k+2\ell=n}} \frac{1}{\sqrt{k}} \ell^{\frac{d-5}{2}} \lesssim  \frac{1}{\sqrt{n}}, \\
d=3 & \Rightarrow &   \sum_{\substack{(k,\ell)\in \N^\star \times \N^\star \\
  k+2\ell=n}} \frac{1}{\sqrt{k}} \ell^{\frac{d-5}{2}} \lesssim  \frac{\ln(n)}{\sqrt{n}}, \\
d\geq 4 & \Rightarrow &   \sum_{\substack{(k,\ell)\in \N^\star \times \N^\star \\
  k+2\ell=n}} \frac{1}{\sqrt{k}} \ell^{\frac{d-5}{2}} \lesssim  n^{\frac{d-4}{2}} \\ 
\end{eqnarray*}
that turn out to be less than the contribution of the second term in the remainder in \eqref{remrem}: 
\begin{equation*}
  \sum_{\substack{(k,\ell)\in \N^\star \times \N^\star \\
  k+2\ell=n}} \frac{1}{k} \ell^{\frac{d-3}{2}}\lesssim  \log(n) n^{\frac{d-3}{2}}.\qedhere
\end{equation*}
\end{proof}

The sum appearing in \eqref{en-pr} and \eqref{en-pr2} can be well approximated using an integral via the following results (the proof are postponed below). 

\begin{lemm}\label{lemco}  
Let us fix $\beta>-1$. Then, as $\alpha \to \infty$, the following holds 
\begin{equation}\label{lemco1}
\int_{0}^{1} \frac{\cos(\alpha \sqrt{t} )}{\sqrt{t}} (1-t)^{\beta}\,\mathrm{d}t=\mathcal{O}\left(\frac{1}{\alpha^{1+\beta}} \right).
\end{equation}
Moreover, for any $\alpha_0>0$, there exists $\ep(\alpha_0,\beta)\in[0,1)$ such that for any $\alpha \geq \alpha_0$ the following holds
\begin{equation}\label{lemco2}
  \left\vert\int_{0}^{1} \frac{\cos(\alpha \sqrt{t} )}{\sqrt{t}} (1-t)^{\beta} \,\mathrm{d}t \right\vert \leq \ep(\alpha_0,\beta)\int_{0}^{1} \frac{(1-t)^{\beta}}{\sqrt{t}} \,\mathrm{d}t.
\end{equation}
\end{lemm}
  
\begin{lemm}\label{lemcos} 
Let $F:[0,+\infty)\rightarrow [0,+\infty)$ be a bounded, differentiable and Lipschitz function.
For any $\beta\geq \frac{-1}{2}$ and any $a\geq 0$, the following asymptotics hold true as $n\rightarrow +\infty$:
\begin{equation}\label{sumcos}
  \sum_{\substack{ (k,\ell)\in\N^\star \times \N^\star  \\ k+2\ell =n }} \frac{F(a\sqrt{k})}{\sqrt{k}} \ell^\beta =  \frac{n^{\beta+\frac{1}{2}}}{2^{\beta+1}} \int_{0}^1 \frac{F(a \sqrt{n t})}{\sqrt{t}} (1-t)^\beta   \mathrm{d}t  
 +\mathcal{O}(n^{\beta}\log(n)).
\end{equation}
Moreover, the bound on the remainder is uniform with respect to the parameter $a$ belonging to a compact subset of $[0,+\infty)$.
\end{lemm}

\subsection{Proof of \eqref{en-bessel}}\label{subs} 

\textbf{Step 1.} An application of \eqref{sumcos} with $a=\sqrt{2}(\hat{x}+\hat{y})$ and $\beta = \frac{d-3}{2} \geq -\frac{1}{2}$ shows
\begin{eqnarray*}
\sum_{\substack{(k,\ell)\in \N^\star \times \N^\star \\ k+2\ell=n } } \frac{\cos((\hat{x}+\hat{y})\sqrt{2k}) }{   \sqrt{k} } \ell^{\frac{d-3}{2}} &=& \frac{n^{\frac{d-2}{2}}}{2^{\frac{d-1}{2}} }\int_0^{1} \frac{\cos((\hat{x}+\hat{y})\sqrt{2nt})}{\sqrt{t}}(1-t)^{\frac{d-3}{2}}\,\mathrm{d}t\\
 & & \qquad +  \mathcal{O}(n^{\frac{d-3}{2}}\log(n) ).
\end{eqnarray*}
A similar formula holds true for $\hat{x}+\hat{y}$ replaced with $\hat{x}-\hat{y}$.
By denoting $c_d:=\frac{1}{\pi^{\frac{d+1}{2} }\sqrt{2} \Gamma \left( \frac{d-1}{2}\right)}$ for simplicity and using Proposition \ref{unif-rem}, we may write
\begin{eqnarray}\label{term-p}
e_{d,n}(x,y)  & = & c_d\frac{n^{\frac{d}{2}-1}}{2^{\frac{d-1}{2}}}  \int_0^{1} \frac{\cos((\hat{x}-\hat{y})\sqrt{2nt})}{\sqrt{t}}(1-t)^{\frac{d-3}{2}}\,\mathrm{d}t       \\\label{term-p2}
& & +(-1)^n c_d\frac{n^{\frac{d}{2}-1}}{2^{\frac{d-1}{2}}}  \int_0^{1} \frac{\cos((\hat{x}+\hat{y})\sqrt{2nt})}{\sqrt{t}}(1-t)^{\frac{d-3}{2}}\,\mathrm{d}t    + \mathcal{O}(n^{\frac{d-3}{2}}\log(n) ).
\end{eqnarray}
Under the conditions $|x|=|y|=1$, note that Lemma \ref{lemcos} ensures that the remainder is uniform with respect to $(x,y)$ due to the inequalities $0\leq \hat{x}-\hat{y}\leq 2$ and $0\leq \hat{x}+\hat{y}\leq 2$.

\textbf{Step 2.} Let us simplify \eqref{term-p}. For any $u\in \R$ and $\beta>-1$, we now recall the following formula involving the Bessel functions (see \cite[line (1.71.6)]{szeg}):
\begin{eqnarray*}
\int_{0}^1 \frac{\cos(u\sqrt{t})}{\sqrt{t}}(1-t)^\beta \mathrm{d}t& = & 
2\int_{0}^1 \cos(u t) (1-t^2)^{\beta} \mathrm{d}t \\
& =& \int_{-1}^1 \cos(ut) (1-t^2)^{\beta} \mathrm{d}t\\
& =&  \sqrt{\pi}\Gamma(\beta+1) \left( \frac{2}{u} \right)^{\beta+\frac{1}{2}} J_{\beta+\frac{1}{2}}(u).
\end{eqnarray*}
Looking at the definition of the normalized Bessel function \eqref{norm-bess} and remembering the equality $\hat{x}-\hat{y}=|x-y|$ (see \eqref{hat2}), we see that the right-hand side of \eqref{term-p} is 
\begin{eqnarray*}
 & & \frac{1}{\pi^{\frac{d+1}{2} }\sqrt{2} \Gamma \left( \frac{d-1}{2}\right)} \times \frac{n^{\frac{d}{2}-1}}{2^{\frac{d-1}{2}}}\times \sqrt{\pi} \Gamma\Big( \frac{d-1}{2} \Big) \widetilde{J}_{\frac{d}{2}-1}(\sqrt{2n} |x-y| )\\
 & & \qquad = \frac{n^{\frac{d}{2}-1}}{(2\pi)^{\frac{d}{2}}}\widetilde{J}_{\frac{d}{2}-1}(\sqrt{2n} |x-y| ).
\end{eqnarray*}

\textbf{Step 3.} Here are two different ways to finally obtain \eqref{en-bessel}. Both exploit the condition $|x-y|\leq 1$ (see \eqref{hat} for the definition of $\hat{x}$ and $\hat{y}$) in the following way:
\begin{equation*}
\hat{x}+\hat{y}=|x+y| =\sqrt{2|x|^2+2|y|^2-|x-y|^2} \geq \sqrt{3}.
\end{equation*}

Firstly, by applying \eqref{lemco1} with $\alpha = (\hat{x}+\hat{y})\sqrt{2n}\gtrsim \sqrt{n}$ and $\beta=\frac{d-3}{2}$, we find that \eqref{term-p2} is 
\begin{equation}\label{fin-rem}
n^{\frac{d}{2}-1}\mathcal{O}\big( n^{-\frac{d-1}{4}}\big)
+\mathcal{O}\big(n^{\frac{d-3}{2}}\log(n)\big)
=\mathcal{O}\big(n^{\frac{d-3}{4}}+n^{\frac{d-3}{2}}\log(n)\big).
\end{equation}
By separating the cases $d=2$ and $d\geq 3$, we see that the last remainder is precisely 
\begin{equation}\label{fin-rem2}
\rem
\end{equation}
namely the one appearing in the right-hand side of \eqref{en-bessel}. Hence, \eqref{en-bessel} is proved.

Another approach would be to remark that the computations in Step 2 and the terms \eqref{term-p} and \eqref{term-p2} lead to 
\begin{equation}\label{seco-term}
e_{d,n}(x,y)=\frac{n^{\frac{d}{2}-1}}{(2\pi)^{\frac{d}{2}}}\Big(
\widetilde{J}_{\frac{d}{2}-1}(\sqrt{2n} |x-y| )+(-1)^n\widetilde{J}_{\frac{d}{2}-1}(\sqrt{2n} |x+y| ) \Big)+\mathcal{O}\big(n^{\frac{d-3}{2}}\log(n))\big).
\end{equation}
Then by using the usual asymptotics of the Bessel functions as a black box for 
$\widetilde{J}_{\frac{d}{2}-1}(\sqrt{2n} |x+y| )$, we recover \eqref{fin-rem}.

\subsection{Proof of \eqref{en-diag}} We may give two arguments following the previous computations.
Since the remainder of \eqref{en-bessel} is uniform with respect to $x$ and $y$ (upon assuming $|x-y|\leq 1$), one may make tend $y\rightarrow x$ to get 
\begin{equation*}
e_{d,n}(x,x) = \frac{n^{\frac{d}{2}-1}}{(2\pi)^{\frac{d}{2}}} \widetilde{J}_{\frac{d}{2}-1}(0)+\rem
\end{equation*}
which is exactly \eqref{en-diag} because of the equality $\widetilde{J}_{\frac{d}{2}-1}(0)=\frac{1}{\Gamma(d/2)}$.

Another argument would be to follow \eqref{term-p}, \eqref{term-p2} and \eqref{fin-rem2} of the last proof that gives us

\begin{equation*}
e_{d,n}(x,x)=\frac{1}{\pi^{\frac{d+1}{2}} 2^{\frac{d}{2}} \Gamma(\frac{d-1}{2})} \underbrace{ \int_{0}^1 \frac{(1-t)^{\frac{d-3}{2}}}{\sqrt{t}} \mathrm{d}t }_{=:I(d)}+\rem.
\end{equation*}
Here, the integral $I(d)$ can be easily computed via a few changes of variables and a Wallis integral:
\begin{equation*}
I(d)=\int_{0}^1 \frac{t^{(d-3)/2}}{\sqrt{1-t}}\mathrm{d}t=
2\int_{0}^1 \frac{w^{d-2}}{\sqrt{1-w^2}}\mathrm{d}t=2
\int_{0}^{\frac{\pi}{2}} \sin^{d-2}(\theta)\,\mathrm{d}\theta = \frac{\sqrt{\pi}\Gamma\left(\frac{d-1}{2} \right)}{\Gamma\left( \frac{d}{2}\right)}.
\end{equation*}
And we again obtain the expected principal term $\frac{n^{\frac{d}{2}-1}}{(2\pi)^{\frac{d}{2}}\Gamma( \frac{d}{2})}$ of $e_{d,n}(x,x)$ in 
\eqref{en-diag}.

\subsection{Proof of \eqref{en-loin}}

Any $c>0$ will be convenient. It is actually in the proof of \eqref{en-proche} that $c$ will have to be chosen under restrictive conditions.

For a constant $\ep_d\in (0,1)$ that will be chosen later, Step 2 in Subsection \ref{subs} shows that $\ep_d e_{d,n}(x,x)-e_{d,n}(x,y)$ equals 
\begin{equation*}
\frac{c_d n^{\frac{d}{2}-1}}{2^{\frac{d-1}{2}}}\left[\ep_d\int_{0}^{1} \frac{1}{\sqrt{t}}(1-t)^{\frac{d-3}{2}}\mathrm{d}t -
\int_{0}^{1} \frac{\cos((\hat{x}-\hat{y})\sqrt{2nt} ) }{\sqrt{t}}(1-t)^{\frac{d-3}{2}}\mathrm{d}t
\right] +o(n^{\frac{d}{2}-1}).
\end{equation*}
The relations \eqref{hat2} imply $(\hat{x}-\hat{y})\sqrt{2n}=|x-y|\sqrt{2n}$ so that from the assumption \eqref{en-loin} one infers 
\begin{equation*}
(\hat{x}-\hat{y})\sqrt{2n} \geq \sqrt{2}c
\end{equation*}
and therefore an application of \eqref{lemco2} yields 
\begin{equation*}
  \int_{0}^{1} \frac{\cos((\hat{x}-\hat{y})\sqrt{2nt} ) }{\sqrt{t}}(1-t)^{\frac{d-3}{2}}\mathrm{d}t \leq \varepsilon \Big(\sqrt{2}c,\frac{d-3}{2}\Big) \int_{0}^{1} \frac{1}{\sqrt{t}}(1-t)^{\frac{d-3}{2}}\mathrm{d}t.
\end{equation*}
By choosing now $\ep_d=\frac{1+\ep\left(\sqrt{2}c,\frac{d-3}{2}\right)}{2}$ which clearly belongs to $(0,1)$, we get
\begin{equation*}
 \frac{\ep_de_{d,n}(x,x)-e_{d,n}(x,y) }{n^{\frac{d}{2}-1}}\geq \frac{c_d}{2^{\frac{d-1}{2}}}\times \underbrace{\frac{1-\ep(\sqrt{2}c,\frac{d-3}{2})}{2}}_{>0} \int_{0}^{1} \frac{1}{\sqrt{t}}(1-t)^{\frac{d-3}{2}}\mathrm{d}t +o(1)
\end{equation*}
and hence $\ep_d e_{d,n}(x,x)-e_{d,n}(x,y)\geq 0$ holds for sufficiently large $n$. Finally, since the same argument could be applied to $\ep_d e_{d,n}(x,x)+e_{d,n}(x,y)$, this concludes the proof of \eqref{en-loin}.

\subsection{Proof of \eqref{en-proche}}

In the following, we make extensive use of the notations $\hat{x}$ and $\hat{y}$ introduced in \eqref{hat} and \eqref{hat2}. It turns out that it is convenient to introduce a new variable $s$ which is equivalent to the distance between $x$ and $y$: 

\begin{lemm}  
Let us consider $d\geq 2$ and $(x, y) \in \S^{d-1}\times \S^{d-1}$. Then the condition $|x-y|\leq 1$ is equivalent to the existence of a real number $s\in \left[0,\frac{\sqrt{3}}{2}\right]$ satisfying 
\begin{equation*}
\hat{y}=\sqrt{1-s} \qquad \mbox{and}\qquad \hat{x}=\sqrt{1+s}.
\end{equation*}
Moreover, $s$ satisfies 
\begin{equation}\label{dist}
s\leq |x-y|\leq \frac{2}{\sqrt{3}}s.
\end{equation}
\end{lemm}

\begin{proof} 
The definition \eqref{hat} implies that $\hat{x} \geqslant |x|= 1$. We also observe $\hat{x}^2+\hat{y}^2=|x|^2+|y|^2=2$, and therefore infer the following 
\begin{equation*}
\hat{x}\in [1,\sqrt{2}] \qquad \text{and}\qquad -1\leq \hat{y}\leq 1.
\end{equation*}
Using the condition $|x-y|\leq 1$ in \eqref{hat2}, we get $2-2\langle x,y\rangle =|x-y|^2 \leq 1$ from which we infer $\hat{x}\hat{y}=\langle x,y\rangle \geq \frac{1}{2}$. It forces $\hat{y}>0$. Now, there exists $s \in [0,1]$ such that $\hat{x}=\sqrt{1+s}$ and since $\hat{y}>0$ it follows that $\hat{y}=\sqrt{1-s}$. We also have 
\begin{equation*}
\sqrt{1-s}\sqrt{1+s}\geq \frac{1}{2} \qquad \text{if and only if}\qquad s\in \left[0,\frac{\sqrt{3}}{2}\right].
\end{equation*}
The inequality \eqref{dist} is now a consequence of \eqref{hat2} and the fact that $\frac{\hat{x}-\hat{y}}{s} = \frac{2}{\sqrt{1+s}+\sqrt{1-s}}$ is increasing with respect to $s$.
\end{proof}

We move to the proof of \eqref{en-proche}, for which our starting point is \eqref{edn}. Let us introduce 
\begin{equation}\label{defo}
\Omega_k(s):=h_k(1)^2-h_k(\sqrt{1-s})h_k(\sqrt{1+s})
\end{equation}
which allows to efficiently rewrite 
\begin{equation}\label{defo2}
e_{d,n}(x,x)-e_{d,n}(x,y)=\frac{1}{\pi^{\frac{d-1}{2}}}\sum_{ \substack{(k,\ell)\in \N\times \N\\ k+2\ell=n}   } \Omega_k(s) B_d(\ell).
\end{equation}
Since $h_0$ is a multiple of $x\mapsto e^{-x^2/2}$, one can check that for any $s$ the following holds  
\begin{equation}\label{defo3}
\Omega_0(s)=0.
\end{equation}
For $k=1$, we have $h_1(x)=\frac{\sqrt{2}}{\pi^{1/4}} xe^{-x^2/2}$ (see \eqref{fs}) and hence the  formula
\begin{equation*}
\Omega_1(s)=\frac{2}{\sqrt{\pi} e}\left(
1-\sqrt{1-s^2}
\right)
\end{equation*}
for which we get the asymptotics $\Omega_1(s)\simeq s^2$ near $s=0$.
This is actually a general fact for any $k\geq 1$. 

\begin{prop}\label{choix} 
There exist positive constants $C,C'$ and $c$ such that for any $k\geq 1$ and any $s\in [0,\frac{c}{\sqrt{k}}]$ the following holds 
\begin{equation*}
C s^2 \leq \frac{\Omega_k(s)}{\sqrt{k}}\leq C' s^2.
\end{equation*}
\end{prop}

The proof of this result is based on the following estimates of higher order derivatives of $\Omega_k$ (see the proof below). 

\begin{prop}\label{omder} 
There exist positive constants $C_1,C_1'$ and $C_2$ such that for any $k\geq 1$ the following holds 
\begin{equation}\label{omeg1}
  C_1 \sqrt{k}\leq  \Omega_k^{(2)}(0)  \leq C_1' \sqrt{k}
\end{equation}
and 
\begin{equation}\label{omeg2}
  \|\Omega_k^{(3)}\|_{L^\infty(0,\frac{\sqrt{3}}{2})}  \leq  C_2 k.
\end{equation}
\end{prop}

\begin{proof}[Proof of Proposition \ref{choix}]
Let us denote $c=\frac{C_1}{2C_2}$, where $C_1$ and $C_2$ are given by Proposition \ref{omder}. The mean value theorem allows us to bound from below as follows: for any $s \in [0,\frac{c}{\sqrt{k}}]$ we may write
\begin{equation*}
\Omega_k^{(2)}(s)\geq \Omega_k^{(2)}(0)-s C_2 k 
\geq (C_1 -c C_2)\sqrt{k}=\frac{C_1}{2} \sqrt{k}, 
\end{equation*}
which gives after two integrations in $s$ that 
\begin{equation*}
  \Omega_k(s)\geq \Omega_k(0)+\Omega_k'(0)s+ \frac{C_1}{4} \sqrt{k} s^2.
\end{equation*}
Furthermore, \eqref{defo} shows $\Omega_k(0)=\Omega_k'(0)=0$. The lower bound is thus obtained. The upper bound is similarly obtained by observing  
\begin{equation*}
\Omega_k^{(2)}(s)\leq \Omega_k^{(2)}(0)+s C_2 k \leq (C_1'+ cC_2)\sqrt{k}.\qedhere
\end{equation*}

\end{proof}

\begin{proof}[Proof of Proposition \ref{omder}]
\textbf{Step 1.} Let us check that \eqref{omeg1} holds asymptotically as $k\to \infty$. To this end we use straightforward computations of the double differentiation of \eqref{defo}, which gives 
\begin{eqnarray*}
 \Omega_k^{'}(s) &= &    \frac{h_k'(\sqrt{1-s})h_k(\sqrt{1+s})}{2\sqrt{1-s}} -\frac{h_k(\sqrt{1-s})h_k'(\sqrt{1+s})}{2\sqrt{1+s}}    \\
  \Omega_k^{(2)}(s) &=& - \frac{1}{4} \frac{h_{k}''(\sqrt{1-s}) h_{k}(\sqrt{1+s})}{1-s}-\frac{1}{4} \frac{h_{k}(\sqrt{1-s})h_{k}''(\sqrt{1+s})}{1+s}\\
 & &  +\frac{1}{4} \frac{h_{k}'(\sqrt{1-s}) h_{k}(\sqrt{1+s})}{(1-s)^{3/2}}      +
 \frac{1}{4} \frac{h_{k}(\sqrt{1-s}) h_{k}'(\sqrt{1+s})}{(1+s)^{3/2}} \\
 & &  +\frac{1}{2} \frac{h_{k}'(\sqrt{1-s})h_{k}'(\sqrt{1+s})}{\sqrt{1-s}\sqrt{1+s}},
\end{eqnarray*}
so that evaluating at $s=0$ gives 
\begin{equation*}
  2\Omega_k^{(2)}(0) = - h_{k}''(1)h_{k}(1) + \big(h_{k}(1)+h_k'(1)\big)h_{k}'(1).
\end{equation*}
We now recall classical formulas on the Hermite functions:
\begin{equation}\label{hn-diff}
-h_n''+x^2h_n=(2n+1)h_n \qquad \mbox{and}\qquad h_n'=\sqrt{2n}h_{n-1}-xh_n.
\end{equation}
This allows the rewriting 
\begin{eqnarray}\nonumber
2\Omega_k^{(2)}(0) &=&  2kh_k(1)^2+\sqrt{2k}h_{k-1}(1)\left( \sqrt{2k}h_{k-1}(1)-h_k(1)\right)\\\label{der2}
&=&  \sqrt{2k}\left[ \sqrt{2k}(h_k(1)^2+h_{k-1}(1)^2)-h_{k-1}(1) h_k(1) \right].
\end{eqnarray}

Using \eqref{hn3} and \eqref{hn1} we obtain 
\begin{equation*}
  \Omega_k^{(2)}(0) \eq{k\rightarrow +\infty} \frac{\sqrt{2}}{\pi}\sqrt{k}.
\end{equation*}

\textbf{Step 2.} Next, in order to check that \eqref{omeg1} is also true for the small integers, it is sufficient to prove $\Omega_k^{(2)}(0)>0$ for any $k\geq 1$. For this end, we notice the elementary inequality
\begin{eqnarray*}
\sqrt{2k}(x^2+y^2)-xy &=&\Big(\sqrt{2k}-\frac{1}{2}\Big)(x^2+y^2)+ \frac{1}{2}(x-y)^2, \qquad  \forall (x,y)\in \R^2 \\
& \geq & \Big(\sqrt{2k}-\frac{1}{2}\Big)(x^2+y^2).
\end{eqnarray*}
As a consequence of \eqref{der2}, we see $\Omega_k^{(2)}(0)\geq 0$.
Moreover, the equality $\Omega_k^{(2)}(0)=0$ would imply $h_k'(1)=h_k(1)=0$ but since $h_k$ is a non-zero function satisfying the second order linear differential equation given in \eqref{hn-diff}, the linear Cauchy theory implies $(h_k(1),h_k'(1))\neq (0,0)$. 
Therefore, we get the positivity of $\Omega_k^{(2)}(0)$ and hence \eqref{omeg1} holds for all $k\geq 1$.

\textbf{Step 3.} It remains to prove \eqref{omeg2}. Writing $u_k(s)=h_k(\sqrt{1+s})$ for $s\in \left[ -\frac{\sqrt{3}}{2},\frac{\sqrt{3}}{2}\right]$, we may differentiate \eqref{defo}: 
\begin{equation*}
|\Omega_k^{(3)}(s)|=|u_k^{(3)}(s) u_k(-s)-3u_k^{(2)}(s) u_k'(-s)+3u_k'(s) u_k^{(2)}(-s)-u_k(s)u_k^{(3)}(-s)|.
\end{equation*}
Owing to the triangle inequality, \eqref{omeg2} will follow from the fact that for $\ell \in \{0,1,2,3\}$ and for all $k \geq 1$ the following inequality holds
\begin{equation}\label{faa}
\sup\limits_{|s|\leq \frac{\sqrt{3}}{2} } |u_k^{(\ell)}(s)| \leq c_\ell k^{(2\ell-1)/4}.
\end{equation}
Note that \eqref{hn1} gives $\|h_k\|_{L^\infty(-M,M)}\lesssim \frac{1}{k^{1/4}}$ for any fixed constant $M>0$ which implies the case $\ell = 0$ of \eqref{faa}. 
 
The cases $\ell \in \{1,2,3\}$ are then the consequence an application of the Fa\`a di Bruno formula for the differentiation of composition of functions as indeed one can observe that the derivatives of $s\mapsto \sqrt{1+s}$ are bounded on the fixed interval $s\in \left[ -\frac{\sqrt{3}}{2},\frac{\sqrt{3}}{2}\right]$ and the fact that \eqref{hn-diff} and the bound in the $\ell = 0$ iteratively imply the bounds $\|h_k^{(\ell)}\|_{L^\infty(-M,M)}\leq c_\ell' k^{(2\ell-1)/4 }$ for suitable positive constants $c_\ell'$.
\end{proof}

We are now ready to prove \eqref{en-proche}. Let us assume that $|x-y|\leq \frac{c}{\sqrt{n}}$, where $c$ is given by Proposition \ref{choix}, in particular $s\leq \frac{c}{\sqrt{n}}$ thanks to \eqref{dist}. 

Let us now use the formula \eqref{defo2}, incorporating \eqref{defo3} and applying Proposition \ref{choix}. This allows us to write 
\begin{equation*}
  e_{d,n}(x,x)-e_{d,n}(x,y) \simeq s^2 \sum_{   \substack{  (k,\ell)\in \N^\star \times \N\\k+2\ell=n}} \sqrt{k} B_d(\ell).
\end{equation*}

From \eqref{dist} and \eqref{Adell} we may bound from below: 
\begin{equation*}
e_{d,n}(x,x)-e_{d,n}(x,y)  \gtrsim |y-x|^2\sum_{ \substack{\lfloor \frac{n}{4}\rfloor \leq \ell \leq \lfloor \frac{n}{3} \rfloor \\ }} (n-2\ell)^{\frac{1}{2}}  \ell^{\frac{d-3}{2}} \gtrsim n^{\frac{d}{2}}|y-x|^2.
\end{equation*}
For the upper bound, the argument is completely similar (upon considering the eventual term $\ell=0$ containing $B_d(0)=1$) as one can write 
\begin{equation*}
e_{d,n}(x,x)-e_{d,n}(x,y) \lesssim |y-x|^2\Big(\sqrt{n}+ \sum_{   \substack{  (k,\ell)\in \N^\star \times \N^\star\\
k+2\ell=n}} \sqrt{k} \ell^{\frac{d-3}{2}} \Big).
\end{equation*}
The conclusion follows by bounding separating the contributions $k\leq \frac{n}{3}$ and $k\geq \frac{n}{3}$ in the above sum in order to obtain 
\begin{equation*}
\sum_{\substack{(k,\ell)\in \N^\star \times \N^\star\\
k+2\ell=n}} \sqrt{k} \ell^{\frac{d-3}{2}}\lesssim n^{\frac{d}{2}}.
\end{equation*}

\appendix

\section{Proof of Theorem \ref{thm.main-harmonic}}\label{anxA}

For the convenience of the reader, we give elements of the proof. We recall the following classical result\footnote{We also refer to \cite[Theorem 1.d.6,
Corollary 1.f.9]{linden2} in the natural context of Banach lattices of finite cotype (Rademacher series and Gaussian series have the same behavior).}.

\begin{prop}\label{proA1}
Let us fix $p\in[1,+\infty)$ and a sequence $(u_n)_{n\in \N}$ in $L^p(\R^d)$. Then the following two statements are equivalent:
\begin{enumerate}[i) ]
\item the function $\sqrt{\sum\limits_{n\in \N} |u_n|^2 }$ belongs to $L^p (\R^d)$.
\item the Gaussian series $\sum g_n(\omega) u_n$ almost surely converges in $L^p(\R^d)$, 
\end{enumerate}
\end{prop}
\begin{proof}
From the Kahane-Khintchine inequalities (see \cite[page 44]{pisier1981}), it is known that ii) is equivalent to the convergence of the Gaussian series $\sum g_n(\omega) u_n$ in $L^q(\Omega,L^p(\R^d))$ (whatever is the choice of $q\in [1,+\infty)$). The natural choice is obviously $q=p$. We then write for any positive integers $n_2\geq n_1$.
\begin{equation*}
\Big\|  \sum_{n=n_1}^{n_2} g_n(\omega) u_n     \Big\|_{L^{p}(\Omega,L^p(\R^d)) }^p
=\int_{\Omega} \int_{\R^d} \Big|  \sum_{n=n_1}^{n_2} g_n(\omega) u_n(x)     \Big|^p \mathrm{d}x \mathrm{d}\mathbf{P}(\omega).
\end{equation*}
The Fubini-Toenlli theorem and the Khintchine inequalities ensure that the previous term is equivalent to 
\begin{equation*}
\int_{\R^d} \Big(\int_{\Omega} \Big|  \sum_{n=n_1}^{n_2} g_n(\omega) u_n(x)     \Big|^2  \mathrm{d}\mathbf{P}(\omega)\Big)^{\frac{p}{2}}\mathrm{d}x
\end{equation*}
which equals:
\begin{equation*}
\int_{\R^d} \Big(\sum_{n=n_1}^{n_2} |u_n(x)|^2 \Big)^{\frac{p}{2}}\mathrm{d}x
\end{equation*}
We then easily obtain the equivalence between i) and ii).
\end{proof}

Let us outline the proof of Theorem \ref{thm.main-harmonic}. Due to Proposition \ref{proA1} and the definition \eqref{deffg} of $f_n^{G,\omega}$, the assertion i) of Theorem \ref{thm.main-harmonic} is equivalent to the following finiteness:
\begin{equation*}
\int_{\R^d} \Big(  \sum_{n\geq 1} \frac{\|f_n\|_{L^2(\R^d)}^2}{\dim(E_n)} \sum_{k=1}^{\dim(E_n)} \varphi_{n,k}(x)^2          \Big)^{\frac{p}{2}} \mathrm{d}x <+\infty,
\end{equation*}
namely
\begin{equation*}
\int_{\R^d} \Big(  \sum_{n\geq 1} \frac{\|f_n\|_{L^2(\R^d)}^2}{\dim(E_n)} e_{d,n}(x,x)       \Big)^{\frac{p}{2}} \mathrm{d}x <+\infty.
\end{equation*}
In the present paper, we were mainly concerned by the spectral function $e_{d,n}$ on $\S^d\times \S^d$ (see Proposition \ref{en}). We however need information about $e_{d,n}(x,x)$ for $x$ belonging to $\R^d$ as in \cite[Proposition 4.1]{imek-crusep}. In order to reach the condition the assertion ii) of Theorem \ref{thm.main-harmonic}, we just have to follow the computations of \cite[pages 340-341]{imek-crusep}.

\section{Proof of Lemma \ref{lemco}}\label{anxB}

We start by showing that \eqref{lemco2} is a consequence of \eqref{lemco1}. To this end, note that the left-hand side of \eqref{lemco2} is a continuous function with respect to $\alpha$, which tends to $0$ as $\alpha\rightarrow +\infty$ thanks to \eqref{lemco1}.

\begin{center}
\includegraphics[scale=0.4]{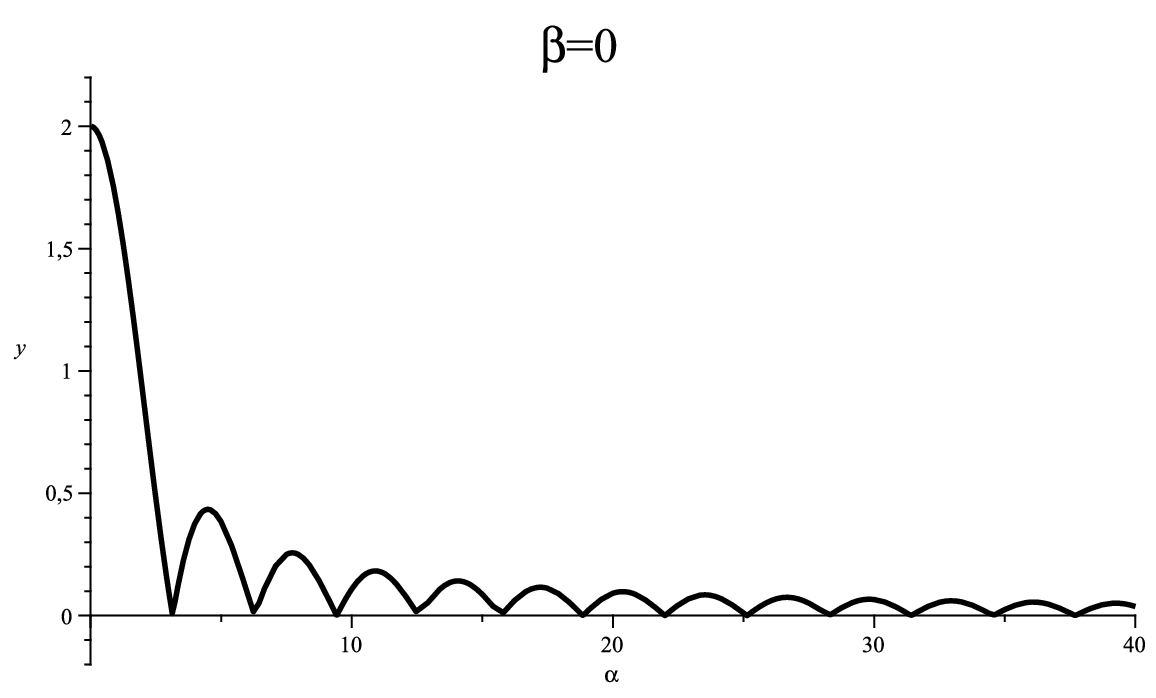} \includegraphics[scale=0.4]{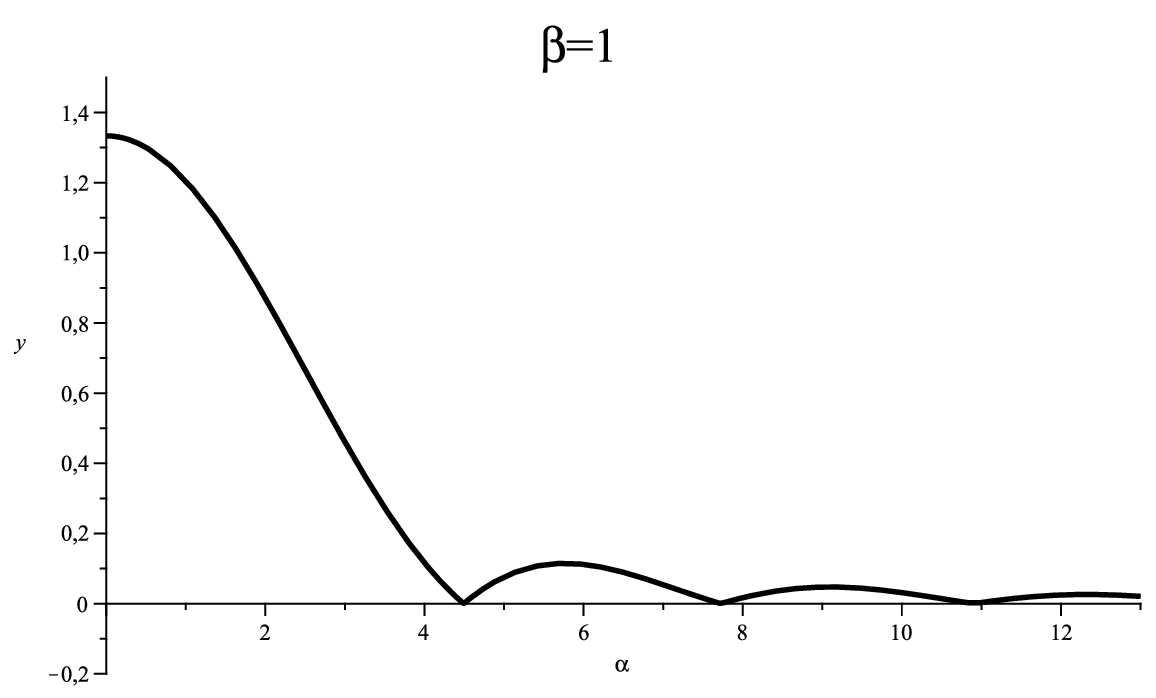}
\end{center}

The triangle inequality shows that the left-hand side of \eqref{lemco2} takes values strictly less than the value for $\alpha=0$:
\begin{equation*}
\int_{0}^{1} \frac{(1-t)^\beta}{\sqrt{t}}\mathrm{d}t.
\end{equation*}
As a consequence of a compactness argument, we obtain \eqref{lemco2} for a suitable constant $\ep(\alpha_0,\beta)\in [0,1)$.

Our main task is now to prove \eqref{lemco1}. First, we change variables $t \leftarrow t^2$ so that 
\begin{eqnarray}\nonumber
  \int_{0}^{1} \frac{\cos(\alpha \sqrt{t} )}{\sqrt{t}} (1-t)^\beta \,\mathrm{d}t & =& 2\int_{0}^{1} \cos(\alpha t)(1-t^2)^\beta  \mathrm{d}t\\ \label{mochebes}
  & =& \int_{-1}^{1} \cos(\alpha t) (1-t^2)^\beta \mathrm{d}t.
\end{eqnarray}

We opt for an elementary argument\footnote{We shall not follow the following way but we have already remarked that \eqref{mochebes} can be expressed with the Bessel functions (see \cite[(1.71.6)]{szeg}):
\begin{equation*}
\sqrt{\pi} \Gamma(\beta+1)\left( \frac{2}{\alpha} \right)^{\beta+\frac{1}{2}}J_{\beta+\frac{1}{2}}(\alpha)
\end{equation*}
and then their asymptotics could be used, as a black box, to derive \eqref{lemco1}.} and start by writing  
\begin{equation}\label{upbo}
\beta=\lceil \beta \rceil -\varepsilon\qquad \mbox{ with }\varepsilon\in [0,1), 
\end{equation}
and $\lceil \beta \rceil$ being the least upper integer of $\beta$.
For any $t\in (-1,1)$, $\rho =\pm 1$ and any couple $(k,\ell)\in \mathbb{N}^2$ satisfying $0\leq k\leq \ell$, we note the elementary differential inequality
\begin{equation*}
| \{(1+\rho t)^{\beta} \}^{(k)}| \leq C_{\beta,\ell} |1+\rho t|^{\beta-\ell}
\end{equation*}
which may be combined to the generalized Leibniz rule in order to get 
\begin{eqnarray}\nonumber
\left\vert\frac{\mathrm{d}^\ell}{\mathrm{d}t^\ell} \{(1-t^2)^\beta\} \right\vert&=& \left\vert\sum_{k=0}^{\ell}\frac{\ell!}{k!(\ell-k)!}\{ (1-t)^{\beta} \}^{(k)}\{ (1+t)^{\beta}\}^{(\ell-k)} \right\vert\\\label{betal}
& \leq & C_{\beta,\ell}(1-t^2)^{\beta-\ell}.
\end{eqnarray}
The general idea is now to perform $\lceil \beta \rceil$ integrations by parts. In the sequel, we shall note $\wp=\pm \cos$ or $\wp=\pm \sin$ so that these integrations by parts read 
  \begin{eqnarray}\label{lafa}
  \int_{-1}^{1} \cos(\alpha t) (1-t^2 )^{\beta} \mathrm{d}t & =&  \left( -\frac{1}{\alpha}\right)^{\lceil \beta \rceil} \int_{-1}^1 \wp'(\alpha t) \frac{\mathrm{d}^{\lceil \beta \rceil }}{\mathrm{d}t^{\lceil \beta \rceil}}  \{ (1-t^2)^\beta\}   \mathrm{d}t.
  \end{eqnarray}
Note that in these integrations by parts, all the boundary terms vanish because \eqref{betal} implies $\frac{d^{\ell}}{dt^{\ell}}  \{ (1-t^2)^\beta\} (\pm 1)=0$ for all integers $\ell <\lceil \beta \rceil $.

Let us us split the integral in right-hand side of \eqref{lafa} as the sum of the two following integrals: 
\begin{eqnarray*}
  \mathcal{I}_\alpha & := & \displaystyle\int_{-1+\frac{1}{\alpha}}^{1-\frac{1}{\alpha}}\wp'(\alpha t) \frac{\mathrm{d}^{\lceil \beta \rceil }}{\mathrm{d}t^{\lceil \beta \rceil}}  \{ (1-t^2)^\beta\}\mathrm{d}t,      \\
  \mathcal{J}_\alpha & := &  \displaystyle\int_{\left[-1,-1+\frac{1}{\alpha}\right]\cup \left[1-\frac{1}{\alpha},1\right] }\wp'(\alpha t) \frac{\mathrm{d}^{\lceil \beta \rceil }}{\mathrm{d}t^{\lceil \beta \rceil}}  \{ (1-t^2)^\beta\}   \mathrm{d}t.
\end{eqnarray*}
Due to the factor $\frac{1}{\alpha^{\lceil \beta \rceil}}$ in \eqref{lafa} and to the formula \eqref{upbo}, proving the bound \eqref{lemco1} boils down to showing  
\begin{equation*}
|\mathcal{I}_\alpha|+|\mathcal{J}_\alpha|\lesssim \frac{1}{\alpha^{1-\ep}}.
\end{equation*}
  
The term $\mathcal{J}_\alpha$ can be bounded by taking advantage of the parity (or imparity) of the function, using the triangle inequality, $|\wp'|\leq 1$ and \eqref{betal}:
\begin{equation*}
  |\mathcal{J}_\alpha| \leq  2\int_{1-\frac{1}{\alpha}}^{1} \left\vert \frac{\mathrm{d}^{\lceil \beta \rceil }}{\mathrm{d}t^{\lceil \beta \rceil}}  \{ (1-t^2)^\beta\}  \right\vert \mathrm{d}t  \lesssim  \int_{1-\frac{1}{\alpha}}^{1} (1-t)^{\beta-\lceil \beta \rceil } \mathrm{d}t.
\end{equation*}
Finally, \eqref{upbo} gives us $|\mathcal{J}_\alpha| \lesssim \frac{1}{\alpha^{1-\ep}}$ for $\alpha \rightarrow +\infty$.
  
The term $\mathcal{I}_\alpha$ is handled using one more integration by parts: 
\begin{equation*}
  \mathcal{I}_\alpha =  \frac{1}{\alpha}\left[\wp(\alpha t) \frac{\mathrm{d}^{\lceil \beta \rceil }}{\mathrm{d}t^{\lceil \beta \rceil}}  \{ (1-t^2)^\beta\}\right]_{-1+\frac{1}{\alpha}}^{1-\frac{1}{\alpha}}-\frac{1}{\alpha} \int_{-1+\frac{1}{\alpha}}^{1-\frac{1}{\alpha}}
   \wp(\alpha t)  \frac{\mathrm{d}^{1+\lceil \beta \rceil }}{\mathrm{d}t^{1+\lceil \beta \rceil}}  \{ (1-t^2)^\beta\}\,\mathrm{d}t. 
\end{equation*}
Using the triangle inequality and incorporating the bound \eqref{betal} with $\ell=\lceil \beta \rceil$, we can write
\begin{eqnarray*}
  |\mathcal{I}_\alpha| & \leq & \frac{C}{\alpha^{1-\ep}}+\frac{C}{\alpha} \int_{-1+\frac{1}{\alpha}}^{1-\frac{1}{\alpha}}
   \left\vert \frac{\mathrm{d}^{1+\lceil \beta \rceil }}{\mathrm{d}t^{1+\lceil \beta \rceil}}  \{ (1-t^2)^\beta\}  \right\vert \,\mathrm{d}t .
   \end{eqnarray*}
If $\ep=0$, namely $\beta=\lceil \beta \rceil$, the last integral involves a polynomial that can be uniformly bounded on $[-1,1]$ and we reach the expected bound $\frac{C}{\alpha}$.

For the last case $0<\ep<1$, we again use \eqref{betal} with $\ell=1+\lceil \beta \rceil$ to conclude   
\begin{eqnarray*}
 |\mathcal{I}_\alpha |  & \leq & \frac{C}{\alpha^{1-\ep}}+ \frac{C}{\alpha}\int_{0}^{1-\frac{1}{\alpha}} (1-t)^{-\ep-1} \,\mathrm{d}t \\
    & \leq & \frac{C}{\alpha^{1-\ep}}.
\end{eqnarray*}
The inequality \eqref{lemco1} is thus proved.

\section{Proof of Lemma \ref{lemcos}}\label{pr-lemcos}

We recall that from the Euler-Maclaurin formula on comparison of series and integrals we can write for any $f\in \mathcal{C}^1(\R,\R)$ and any integers $q\geq p$: 
\begin{equation*}
\left\vert \sum_{\ell=p}^{q} f(\ell)- \int_{p}^{q} f(t)\,\mathrm{d}t\right\vert \leq \frac{|f(p)|+|f(q)|}{2}+\frac{1}{2} \int_{p}^q |f'(t)|\,\mathrm{d}t.
\end{equation*}

\textbf{Step 1.} We first explain the proof of the following asymptotics: 

\begin{equation}\label{anxC0}
\Big\vert\sum_{\substack{ (k,\ell)\in\N^\star \times \N^\star  \\ k+2\ell =n }} \frac{F(a\sqrt{k})}{\sqrt{k}} \ell^\beta -\int_{1}^{\lfloor \frac{n-1}{2} \rfloor} \frac{F(a\sqrt{n-2t})}{\sqrt{n-2t}} t^\beta \mathrm{d}t \Big\vert \lesssim n^\beta \log(n).
\end{equation}

For any $t\in \big(0,\frac{n}{2}\big)$, we define $f(t)=\frac{F(a\sqrt{n-2t})}{\sqrt{n-2t}} t^\beta $ and differentiate 
\begin{equation*}
f'(t)= -\frac{aF'(a\sqrt{n-2t})}{n-2t}t^\beta + \frac{F(a\sqrt{n-2t})}{\sqrt{n-2t}} \beta t^{\beta-1}+\frac{F(a\sqrt{n-2t})}{(n-2t)^{\frac{3}{2}}}t^\beta.
\end{equation*}
Since $F$ is assumed to be bounded and Lipschitz and $a$ belongs to a compact subset of $[0,+\infty)$, we obtain
\begin{equation*}
|f'(t)|\lesssim  \frac{t^\beta}{n-2t} + \frac{ t^{\beta-1}}{\sqrt{n-2t}} +\frac{t^\beta}{(n-2t)^{\frac{3}{2}}}.
\end{equation*}
The left-hand side of \eqref{sumcos} reads $\displaystyle\sum\limits_{\ell=1}^{\lfloor \frac{n-1}{2}\rfloor} f(\ell)$, so that the above considerations lead to 
\begin{multline*}
\Big\vert\sum_{\substack{ (k,\ell)\in\N^\star \times \N^\star  \\ k+2\ell =n }} \frac{F(a\sqrt{k})}{\sqrt{k}} \ell^\beta -\int_{1}^{\lfloor \frac{n-1}{2} \rfloor} \frac{F(a\sqrt{n-2t})}{\sqrt{n-2t}} t^\beta \mathrm{d}t \Big\vert \\
\lesssim \frac{1}{\sqrt{n}}+n^\beta+ \int_{1}^{\lfloor \frac{n-1}{2} \rfloor} \left(\frac{t^\beta}{n-2t}+\frac{t^{\beta-1}}{\sqrt{n-2t}}+\frac{t^\beta}{(n-2t)^{\frac{3}{2}}}\right) \mathrm{dt}.
\end{multline*}
Remembering the condition $\beta \geq - \frac{1}{2}$ and cutting the integral following $t\leq \frac{n}{4}$ and $t\geq \frac{n}{4}$, we may control the upper bound as follows to get the expected result:
\begin{equation*}
  \mathcal{O}\left( \frac{1}{\sqrt{n}} +n^\beta+ n^\beta \log(n)+\Big(\frac{1}{\sqrt{n}}\int_{1}^{\frac{n}{4}} t^{\beta-1} \mathrm{d}t +n^{\beta-\frac{1}{2}}\Big)+  n^\beta \right) = \mathcal{O}(n^{\beta}\log n) .
\end{equation*}


\textbf{Step 2.} Looking at \eqref{sumcos} and \eqref{anxC0}, it remains to prove the following asymptotic
\begin{equation}\label{anxC}
\int_{1}^{\lfloor \frac{n-1}{2} \rfloor} \frac{F(a\sqrt{n-2t})}{\sqrt{n-2t}} t^\beta \mathrm{d}t
- \frac{n^{\beta+\frac{1}{2}}}{2^{\beta+1}} \int_{0}^1 \frac{F(a \sqrt{n t})}{\sqrt{t}} (1-t)^\beta   \mathrm{d}t =\mathcal{O}(n^\beta \log(n)).
\end{equation}
After making the change of variable $   t\leftarrow 1-\frac{2t}{n} $, the left-hand side becomes
\begin{equation*}
\int_{1}^{\lfloor \frac{n-1}{2} \rfloor} \frac{F(a\sqrt{n-2t})}{\sqrt{n-2t}} t^\beta \mathrm{d}t
-\int_{0}^{\frac{n}{2}} \frac{F(a\sqrt{n-2t})}{\sqrt{n-2t}} t^\beta \mathrm{d}t
\end{equation*}
which is clearly equal to
\begin{equation}\label{anxC2}
 -\int_{0}^{1} 
\frac{F(a\sqrt{n-2t})}{\sqrt{n-2t}} t^\beta \mathrm{d}t-\int_{\lfloor \frac{n-1}{2} \rfloor}^{\frac{n}{2}}\frac{F(a\sqrt{n-2t})}{\sqrt{n-2t}} t^\beta \mathrm{d}t.
\end{equation}

The first integral in \eqref{anxC2} is immediately bounded as follows because $F$ is bounded and $\beta$ satisfies $\beta>-1$:
\begin{eqnarray*}
  \left\vert\int_{0}^{1} 
  \frac{F(a\sqrt{n-2t})}{\sqrt{n-2t}} t^\beta \mathrm{d}t  \right\vert&\lesssim & 
 \frac{1}{\sqrt{n-2}} \int_{0}^{1} 
 t^\beta \mathrm{d}t =\mathcal{O}\Big( \frac{1}{\sqrt{n}}\Big).
\end{eqnarray*}
For the second integral in \eqref{anxC2}, we bound and compute 
\begin{eqnarray*}
 \left\vert \int_{\lfloor \frac{n-1}{2} \rfloor}^{\frac{n}{2}}\frac{F(a\sqrt{n-2t})}{\sqrt{n-2t}} t^\beta \mathrm{d}t \right\vert & \lesssim & \int_{\frac{n-4}{2}}^{\frac{n}{2}}
 \frac{t^\beta}{\sqrt{n-2t}}\mathrm{d}t \\
 & \lesssim & n^\beta \int_{\frac{n-4}{2}}^{\frac{n}{2}} \frac{\mathrm{d}t}{\sqrt{n-2t}}=
2n^\beta.
\end{eqnarray*}
The last two estimates imply \eqref{anxC} thanks to the condition $\beta\geq -\frac{1}{2}$.

\bibliographystyle{alpha}
\bibliography{imekraz-bib.bib}

\end{document}